\documentclass[a4paper]{amsart}

\usepackage{amsmath,amsfonts,amssymb,amsthm} %AMS packages
\usepackage[all, cmtip]{xy} %For commutative diagrams
\usepackage[T1]{fontenc} %Extra font characters

\numberwithin{figure}{section}  
\numberwithin{equation}{section}
\numberwithin{table}{section}

\theoremstyle{definition}
\newtheorem{definition}{Definition}[section]
\newtheorem{remark}[definition]{Remark}
\newtheorem{example}[definition]{Example}

\theoremstyle{plain}
\newtheorem{lemma}[definition]{Lemma}
\newtheorem{theorem}[definition]{Theorem}
\newtheorem{proposition}[definition]{Proposition}
\newtheorem{conjecture}[definition]{Conjecture}

\newcommand{\bA}{\mathbb{A}}
\newcommand{\bC}{\mathbb{C}}
\newcommand{\bF}{\mathbb{F}}
\newcommand{\bG}{\mathbb{G}}
\newcommand{\bN}{\mathbb{N}}
\newcommand{\bP}{\mathbb{P}}
\newcommand{\bQ}{\mathbb{Q}}
\newcommand{\bR}{\mathbb{R}}
\newcommand{\bZ}{\mathbb{Z}}
\newcommand{\WP}{\mathbb{W}\mathbb{P}}

\newcommand{\calA}{\mathcal{A}}
\newcommand{\calC}{\mathcal{C}}
\newcommand{\calD}{\mathcal{D}}
\newcommand{\calE}{\mathcal{E}}
\newcommand{\calL}{\mathcal{L}}
\newcommand{\calM}{\mathcal{M}}
\newcommand{\calN}{\mathcal{N}}
\newcommand{\calO}{\mathcal{O}}
\newcommand{\calP}{\mathcal{P}}
\newcommand{\calV}{\mathcal{V}}
\newcommand{\calY}{\mathcal{Y}}

\DeclareMathOperator{\NS}{NS}
\DeclareMathOperator{\Pic}{Pic}
\DeclareMathOperator{\Hom}{Hom}
\DeclareMathOperator{\Spec}{Spec}

\DeclareMathOperator{\Div}{Div}
\DeclareMathOperator{\Gr}{Gr}
\DeclareMathOperator{\Nef}{Nef}
\DeclareMathOperator{\Perf}{Perf}
\DeclareMathOperator{\Fuk}{Fuk}

\begin{document}

\title[Mirror Symmetry for del Pezzo Surfaces]{Mirror Symmetry for Lattice Polarized del Pezzo Surfaces}

\author[C. F. Doran]{Charles F. Doran}
\address{Department of Mathematical and Statistical Sciences, 632 CAB, University of Alberta,  Edmonton, AB, T6G 2G1, Canada}
\email{charles.doran@ualberta.ca}
\thanks{C. F. Doran (University of Alberta) was supported by the Natural Sciences and Engineering Research Council of Canada, the Pacific Institute for the Mathematical Sciences, and the Visiting Campobassi Professorship at the University of Maryland.}

\author[A. Thompson]{Alan Thompson}
\address{Mathematics Institute, Zeeman Building, University of Warwick,  Coventry, CV4 7AL, United Kingdom \newline \indent
DPMMS, Centre for Mathematical Sciences, Wilberforce Road, Cambridge, CB3 0WB, United Kingdom }
\email{amt69@cam.ac.uk}
\thanks{A. Thompson (University of Warwick/University of Cambridge) was supported by the Engineering and Physical Sciences Research Council programme grant \emph{Classification, Computation, and Construction: New Methods in Geometry}.}

\date{}

\begin{abstract}
We describe a notion of lattice polarization for rational elliptic surfaces and weak del Pezzo surfaces, and describe the complex moduli of the former and the K\"{a}hler cone of the latter. We then propose a version of mirror symmetry relating these two objects, which should be thought of as a form of Fano-LG correspondence. Finally, we relate this notion to other forms of mirror symmetry, including Dolgachev-Nikulin-Pinkham mirror symmetry for lattice polarized K3 surfaces and the Gross-Siebert program.
\end{abstract}


\maketitle

\section{Introduction}

The aim of this paper is to describe a form of mirror symmetry relating weak del Pezzo surfaces to rational elliptic surfaces.

Weak del Pezzo surfaces are $2$-dimensional \emph{weak Fano varieties}, i.e. smooth varieties with nef and big anticanonical divisors. In physics, Eguchi, Hori, and Xiong \cite{gqc} postulated that such varieties are mirror to \emph{Landau-Ginzburg models}, which in the $2$-dimensional setting take the form of elliptic fibrations over $\bC$. This version of mirror symmetry is known as the \emph{Fano-LG correspondence}, and has been widely studied in the mathematical literature. In the surface case, such a study has been performed by Auroux, Katzarkov, and Orlov \cite{msdpsvccs}, who find that the Landau-Ginzburg model of a weak del Pezzo surface can be realized as an open set in a certain rational elliptic surface.

In \cite{mstdfcym}, Doran, Harder, and Thompson conjecture a relationship between the Fano-LG correspondence and classical mirror symmetry for Calabi-Yau varieties, based on earlier ideas of Tyurin \cite{fvcy}. Stated briefly, their conjecture claims that if one degenerates a Calabi-Yau variety to a union of two (quasi-)Fano varieties, glued along a smooth anticanonical divisor, then the mirror Calabi-Yau variety can be constructed by gluing together the Landau-Ginzburg models associated to the (quasi-)Fanos. This conjecture has been proved by Kanazawa in the cases of elliptic curves \cite{dhtcsyzmsec} and Abelian surfaces \cite{dtfgqms}, using SYZ mirror symmetry, but the conjecture remains open for K3 surfaces and higher dimensional Calabi-Yau varieties. Despite this, however, there have been some attempts to use the Doran-Harder-Thompson conjecture to construct new mirrors: Lee \cite{mpcytmpqft} has used a version of it  to construct new candidate mirror pairs of Calabi-Yau threefolds, using techniques from toric geometry, and Kanazawa \cite{dtfgqms} has discussed the possibility of using it to construct new Landau-Ginzburg models.

The motivation for this paper comes from thinking about the $2$-dimen\-sional version of the Doran-Harder-Thompson conjecture. Classical mirror symmetry for $2$-dimensional Calabi-Yau varieties is given by Dolgachev's, Nikulin's, and Pinkham's \cite{sedeask3,fagkk3s,iqfaag,mslpk3s} notion of mirror symmetry for lattice polarized K3 surfaces. This is discussed in \cite[Section 4]{mstdfcym}, where it is postulated that there should be a corresponding notion of lattice polarized mirror symmetry for quasi-Fano surfaces and their Landau-Ginzburg models, that is compatible under the Doran-Harder-Thompson conjecture with mirror symmetry for lattice polarized K3 surfaces.

This paper aims to put this idea on a more rigorous footing. We describe a notion of lattice polarization for weak del Pezzo surfaces and a certain class of elliptic fibrations over $\bC$ (those which admit compactifications to rational elliptic surfaces with a certain fibre type at infinity), then state a mirror conjecture relating these objects. The remainder of the paper is concerned with providing a series of justifications for this conjecture. In particular, we compare it with Dolgachev-Nikulin-Pinkham mirror symmetry for lattice polarized K3 surfaces and prove a number of compatibility results between the two, although unfortunately these do not extend as far as a full proof of the Doran-Harder-Thompson conjecture  in this case (Conjecture \ref{conj:dht}).

We conclude this portion of the introduction by discussing the relationship between our construction and previous work in this area. As mentioned above, mirror symmetry for weak del Pezzo surfaces (without lattice polarization) has previously been studied in the setting of homological mirror symmetry by Auroux, Katzarkov, and Orlov \cite{msdpsvccs}, albeit in the opposite direction: they are primarily concerned with comparing the complex structure of weak del Pezzo surfaces to the symplectic structure of their Landau-Ginzburg models, whereas we mostly work the other way. \cite{msdpsvccs} postulates that the mirror of a weak del Pezzo surface of degree $d$ should be an elliptic fibration over $\bC$ with $(12-d)$ singular fibres, all of which have Kodaira type $\mathrm{I}_1$. This agrees with our conjecture \emph{up to deformation}: in our setting, different choices of lattice polarization on the weak del Pezzo surface will lead to different configurations of singular fibres on its Landau-Ginzburg model, but all configurations may be deformed to the mirror as described by \cite{msdpsvccs}. A more detailed comparison between our conjecture and the work of \cite{msdpsvccs} is given in Remark \ref{rem:ako}.

\subsection{Structure of the paper}

We begin by reviewing some necessary lattice theory and notation in Section \ref{sec:lattices}. Following that, in Section \ref{sec:ratellsurf} we embark upon a study of rational elliptic surfaces. We describe a number of special classes of rational elliptic surfaces, which we call \emph{rational elliptic surfaces of type $d$}, and describe how to endow them with lattice polarizations. Finally, we embark on a detailed study of the complex moduli of our lattice polarized rational elliptic surfaces.

In Section \ref{sec:dp}, we switch to the other side of the mirror and study weak del Pezzo surfaces. We once again describe how to endow such surfaces with a lattice polarization, and prove some results about the structure of the K\"{a}hler cone.

Section \ref{sec:mirror} contains the main conjecture of the paper (Conjecture \ref{conj:mirror}), which postulates a mirror correspondence between the lattice polarized weak del Pezzo surfaces defined in Section \ref{sec:dp} and Landau-Ginzburg models derived from the lattice polarized rational elliptic surfaces described in Section \ref{sec:ratellsurf}. We then proceed to justify this correspondence by comparing the structure of the K\"{a}hler cones of the weak del Pezzos to the structure of the complex moduli spaces of the rational elliptic surfaces.

The final two sections of the paper are concerned with the relation between our mirror symmetry conjecture and other forms of mirror symmetry extant in the literature. We begin in Section \ref{sect:K3} by proving a number of compatibility results between our mirror symmetry conjecture and Dolgachev-Nikulin-Pinkham mirror symmetry for lattice polarized K3 surfaces. Our hope is that these results may form a starting point for an attack on the full Doran-Harder-Thompson conjecture in the $2$-dimensional setting.

Finally, Section \ref{sec:gs} discusses the compatibility between our conjecture and the Gross-Siebert program. We perform explicit Gross-Siebert computations for weak del Pezzo surfaces of high degree (close to the maximum value of $9$), and compare the output with the predictions made by our conjecture. These computations lead us to a final conjecture (Conjecture \ref{conj:GSconj}), which describes how we expect our ideas to fit with the Gross-Siebert program.

\subsection{Acknowledgements}

The authors would like to thank Mark Gross for several helpful discussions about the Gross-Siebert program and Andrew Harder for numerous insightful comments on various drafts of this paper.

\section{Some lattice theory}\label{sec:lattices}

We begin with a review of some necessary lattice theory. For $n \geq 0$, the lattice $\mathrm{I}_{1,n}$ is the unique odd indefinite unimodular lattice of signature $(1,n)$. It is generated by classes $l, e_1,\ldots,e_n$ with intersection numbers $l^2 = 1$, $e_i^2 = -1$ and $l.e_i = e_i.e_j = 0$ for all $i \neq j$. We will let $f_{9-n}$ denote the special class $f_{9-n} := 3l - \sum_{i=1}^n e_i$; note that $f_{9-n}$ has self-intersection $f^2_{9-n} = 9 - n$ (hence the name) and $f_{9-n} . e_i = 1$ for all $i$.

The lattice $\mathrm{II}_{1,1}$ is the unique even indefinite unimodular lattice of signature $(1,1)$ (the \emph{hyperbolic plane}). It is generated by classes $a, b$ with intersection numbers $a^2 = b^2 = 0$ and $a.b = 1$. We denote the special class $2a + 2b$ by $f_{8'}$; note that $f_{8'}^2 = 8$.

Next, let $L$ denote any lattice. A \emph{root} in $L$ is a class $\alpha \in L$ with $\alpha^2 = -2$. A negative definite lattice that is generated by its roots is called a \emph{root lattice}. If $R$ is a root lattice, a set of \emph{positive roots} in $R$ is a subset $\Phi^+$ of the set of all roots in $R$, chosen such that
\begin{itemize}
\item for any root $\alpha \in R$, exactly one of $\alpha$, $-\alpha$ is in $\Phi^+$, and
\item for any two distinct roots $\alpha, \beta \in \Phi^+$ such that $\alpha + \beta$ is a root, $\alpha + \beta \in \Phi^+$.
\end{itemize}
Note that there may be many different ways to choose a set of positive roots in $R$.

Given a set of positive roots in $R$, a \emph{simple root} is a root $\alpha \in \Phi^+$ that cannot be written as a sum of two other roots in $\Phi^+$. Standard results say that the set of simple roots forms a basis for $R$, and that all elements of $\Phi^+$ are linear combinations of simple roots with positive coefficients. Note that the set of simple roots in $R$ is dependent upon the choice of positive roots $\Phi^+$.

There is a well-known classification of root lattices in terms of their associated Coxeter diagrams, defined as follows: vertices in the Coxeter diagram correspond to simple roots $\alpha_i$, two vertices are joined by an edge if $\alpha_i.\alpha_j >0$, and edges are labelled by integers $n_{ij}$ defined such that $2\cos(\frac{\pi}{n_{ij}}) = \alpha_i.\alpha_j$. With this definition, the classification states that the Coxeter diagram corresponding to any root lattice is a union of simply-laced Dynkin diagrams. We will often identify a root lattice by the label ($A_n$, $D_n$, $E_n$) given to its corresponding Dynkin diagram.

By the same procedure, one may also define \emph{affine root lattices} to be lattices associated to affine (extended) Dynkin diagrams. Affine root lattices are degenerate, so the above discussion does not apply, but they do share many important properties with root lattices. We will usually identify affine root lattices by the labels ($\tilde{A}_n$, $\tilde{D}_n$, $\tilde{E}_n$) given to their corresponding affine Dynkin diagrams.

Finally, we will also need to know a little about the roots in $\mathrm{I}_{1,9}$. In this lattice, there is a linearly independent set of nine roots
\begin{align*}\alpha_i &= e_i - e_{i+1} \text{ for } 1 \leq i \leq 8, \\
\alpha_0 &= l - e_1 - e_2 - e_3. & \end{align*}
These roots have the important property that $\alpha_i.f_0 = 0$ for all $i$. The Coxeter diagram associated to this set of roots has type $\tilde{E}_8$. 

\section{Lattice polarized rational elliptic surfaces}\label{sec:ratellsurf}

We begin this section by recalling the theory of the moduli of rational elliptic surfaces. We treat such surfaces as \emph{Looijenga pairs} $(Y,D)$, where $Y$ is a rational surface and $D \in |-K_Y|$ is a cycle of rational curves. The study of such pairs was intiated by Looijenga in \cite{rsac}, who performed a detailed study in the setting where the cycle $D$ has $\leq 5$ components. Looijenga's results were extended to arbitrary numbers of components by Gross, Hacking, and Keel \cite{msacc}, whose exposition we follow below.

Let $Y$ denote a rational elliptic surface (in other words, $Y$ is a rational surface equipped with a morphism $Y \to \bP^1$, whose general fibre is an elliptic curve). Recall that the Picard lattice $\Pic(Y)$ is isomorphic to $\mathrm{I}_{1,9}$. Following  \cite[Definition 1.7]{msacc}, let $\calC^+_Y$ denote the connected component of the cone $\{\beta \in \Pic(Y)\otimes \bR \mid \beta^2 > 0\}$ which contains the ample classes. Define $\calE \subset \Pic(Y)$ to be the collection of classes $E$ with $E^2 = K_Y.E = -1$ and $E.H > 0$ for a given ample $H$; by \cite[Lemma 2.13]{msacc}, this is independent of the choice of $H$. Then let $\calC^{++}_Y \subset \calC^+_Y$ be the subcone consisting of those $\beta \in \calC^+_Y$ satisfying the inequalities $\beta.E \geq 0$ for all $E \in \calE$.

We next define a special type of rational elliptic surface.

\begin{definition} Let $1 \leq d \leq 9$ be an integer. A \emph{rational elliptic surface of type $d$} is a pair $(Y,D)$ consisting of a rational elliptic surface $Y$ and a fibre $D \subset Y$ of Kodaira type $\mathrm{I}_d$.
\end{definition}

\begin{remark} Persson \cite{ckfres} has classified the possible configurations of Kodaira fibres on a rational elliptic surface. A consequence of his classification is that there are two distinct deformation types of rational elliptic surfaces containing a fibre of type $\mathrm{I}_8$. These correspond to the two different ways of embedding an $\mathrm{I}_8$ fibre into the Picard lattice $\mathrm{I}_{1,9}$ or, equivalently, to the two different embeddings of the $A_7$ root lattice into the affine root lattice $\tilde{E}_8 \subset I_{1,9}$. We call these two possibilities \emph{rational elliptic surfaces of type $8$}, which correspond to the embedding $A_7 \subset \tilde{E}_8$ given by $\{\alpha_2,\alpha_3,\ldots,\alpha_8\}$, and \emph{rational elliptic surfaces of type $8'$}, which correspond to the embedding $A_7 \subset \tilde{E}_8$ given by $\{\alpha_0,\alpha_3,\ldots,\alpha_8\}$. One may also distinguish these possibilities by examining the torsion subgroup of the Mordell-Weil group: this subgroup is trivial for rational elliptic surfaces of type $8$, but nontrivial for rational elliptic surfaces of type $8'$.
\end{remark}

Note that if $(Y,D)$ is a rational elliptic surface of type $d$, then $D \in |-K_Y|$ is a cycle of rational curves, so $(Y,D)$ is an example of a Looijenga pair. Denote the components of $D$ by $D_1,D_2,\ldots,D_d$, cyclically ordered.

\begin{definition}\label{def:marking} A \emph{marking} of a rational elliptic surface $(Y,D)$ of type $d$, for $1 \leq d \leq 9$, is a choice of isometry $\mu \colon \Pic(Y) \to \mathrm{I}_{1,9}$ satisfying
\[\mu(D_i) = \delta_i := \left\{ \begin{array}{ll} \alpha_{9-i} & \text{for } 1 \leq i \leq d-1 \\
f_0 - \alpha_8 - \alpha_7 - \cdots - \alpha_{10-d} & \text{for } i=d,  \end{array} \right.\]
with the additional compatibility condition that if $(Y_0,D_0,\mu_0)$ is a (fixed) generic marked rational elliptic surface of type $d$, then $\mu(\calC^{++}_Y) = \mu_0(\calC^{++}_{Y_0})$. Denote the cone $\mu_0(\calC^{++}_{Y_0})$ by $\calC^{++}_d \subset \mathrm{I}_{1,9}$.

 If $d = 1$, by convention we have $\mu(D_1) = \delta_1 := f_0$.
\end{definition}

However, this is not quite the complete picture; we also need to define a marking on a rational elliptic surface of type $8'$.

\begin{definition}\label{def:marking'} A \emph{marking} of a rational elliptic surface $(Y,D)$ of type $8'$ is a choice of isometry $\mu \colon \Pic(Y) \to \mathrm{I}_{1,9}$ satisfying
\[\mu(D_i) = \delta_i := \left\{ \begin{array}{ll} \alpha_{9-i} & \text{for } 1 \leq i \leq 6 \\
\alpha_0 & \text{for } i=7 \\
f_0 - \alpha_0 - \alpha_8 - \alpha_7 - \cdots - \alpha_3 & \text{for } i=8,  \end{array} \right.\]
with the additional compatibility condition that if $(Y_0,D_0,\mu_0)$ is a (fixed) generic marked rational elliptic surface of type $8'$, then $\mu(\calC^{++}_Y) = \mu_0(\calC^{++}_{Y_0})$. Denote the cone $\mu_0(\calC^{++}_{Y_0})$ by $\calC^{++}_{8'} \subset \mathrm{I}_{1,9}$.
\end{definition}

\begin{definition} Let $F_d$ (resp. $F_{8'}$) denote the sublattice of $\mathrm{I}_{1,9}$ generated by the classes $\delta_i$ from Definition \ref{def:marking} (resp. Definition \ref{def:marking'}).  
\end{definition}

In general, when no confusion is likely to result, we will simply refer to markings of rational surfaces of type  $d$, lattices $F_d$, and cones $\calC_d^{++}$, with the understanding that this includes the case $d = 8'$ unless otherwise specified. The orthogonal complement of $F_d$ in $\mathrm{I}_{1,9}$ will be denoted $F_d^{\perp}$; the lattices $F_d^{\perp}$ have rank $10 - d$, with intersection form given by Table \ref{tab:ellipticlatticetype}.

\begin{table}
\begin{tabular}{|c|c|c|}
\hline
$d$ & $F_d^{\perp}$ & $F_d^{\perp}/\langle f_0 \rangle$ \\
\hline 
\rule{0pt}{3ex} $1$ & $\tilde{E}_8$ & $E_8$ \\
$2$ & $\tilde{E}_7$ & $E_7$ \\

$3$ & $\tilde{E}_6$ & $E_6$\\
$4$ & $\tilde{D}_5$ & $D_5$ \\

$5$ & $\tilde{A}_4$ & $A_4$ \\
$6$ & $\tilde{A}_2 + A_1$ & $A_2 + A_1$ \\
$7$ & $\left(\begin{array}{ccc} -2 & 1 & 1 \\ 1 & -4 & 3 \\ 1 & 3 & -4 \end{array}\right)$ & $\left(\begin{array}{cc} -2 & 1 \\ 1 & -4 \end{array}\right)$ \\
$8$ & $\left(\begin{array}{cc} -8 & 8 \\ 8 & -8 \end{array}\right)$ & $(-8)$ \\
$8'$ & $\tilde{A}_1$ & $A_1$ \\
$9$ & $(0)$ & $\{0\}$ \\
\hline
\end{tabular}
\caption{Lattices $F_d^{\perp}$}
\label{tab:ellipticlatticetype}
\end{table}

It follows from the results of \cite{msacc} that there is a period map from the moduli space of marked rational elliptic surfaces of type $d$ to $\Hom(F_d^{\perp},\bG_m)$, which is injective in a neighbourhood of a generic point (a more precise statement will be made in the next section, see Theorem \ref{thm:ratellmoduli}). Indeed, to a marked rational elliptic surface $(Y,D,\mu)$ of type $d$, we may associate the period point $\phi \in \Hom(F_d^{\perp},\bG_m)$ defined by
\begin{align*}
\phi\colon F_d^{\perp} &\longrightarrow \Pic^0(D) \cong \bG_m \\
\alpha &\longmapsto \mu^{-1}(\alpha)|_{D}.
\end{align*} 

Note that $f_0 \in F_d^{\perp}$ for all $d$ and $\mu^{-1}(f_0) \in \Pic(Y)$ is the class of a fibre, so $\mu^{-1}(f_0)|_D \cong \calO_D$. Thus the period point $\phi$ of any rational elliptic surface must satisfy $\phi(f_0) = 1$, so $\phi$ is completely determined by its action on the quotient $F_d^{\perp}/\langle f_0 \rangle$. The quotient $F_d^{\perp}/\langle f_0 \rangle$ is a negative definite lattice of rank $9-d$, with intersection form given by Table \ref{tab:ellipticlatticetype}.

\subsection{Lattice polarizations and moduli} \label{sec:ratellpol} Next we add the concept of a \emph{lattice polarization}.  Let $L \subset F_d^{\perp}$ be a negative definite primitive sublattice and let $R_L$ denote the sublattice of $L$ generated by the roots in $L$. Since $L$ is negative definite, $R_L$ is a root lattice and $L$ is isomorphic to its image in $F_d^{\perp}/\langle f_0 \rangle$.

\begin{definition}   An \emph{$L$-polarization} on a rational elliptic surface $(Y,D)$ of type $d$ is a primitive embedding $\nu\colon L \hookrightarrow \Pic(Y)$ such that
\begin{itemize}
\item $\nu(\beta).D_i = 0$ for all $\beta \in L$ and all $1 \leq i \leq d$, and
\item there exists a set of positive roots $\Phi^+$ in $R_L$ so that $\nu(\Phi^+)$ is contained in the effective cone in $\Pic(Y)$.
\end{itemize}
\end{definition}

\begin{remark} \label{rem:ellipticroot} Note that the important part of this definition really only depends upon the root lattice $R_L$, not on the full lattice $L$. For most purposes, it suffices to take $L \subset F_d^{\perp}$ to be a root lattice, so that $L = R_L$. However, the greater generality in the definition above works better with mirror symmetry; this will be explored later in this paper.
\end{remark}

This definition is all well and good, but to make statements about moduli we need something stronger. For the remainder of this section, assume that we have fixed a primitive embedding $L \hookrightarrow F_d^{\perp}$ of a negative definite lattice $L$ into $F_d^{\perp}$ and chosen a set of positive roots $\Phi_L^+ \subset R_L$.

\begin{definition} A \emph{marked $L$-polarization} on a rational elliptic surface $(Y,D)$ of type $d$ is a choice of marking $\mu \colon \Pic(Y) \to \mathrm{I}_{1,9}$ such that the preimage $\mu^{-1}(\Phi^+_L)$ is contained in the effective cone in $\Pic(Y)$ (i.e. so that $\mu^{-1}|_L$ is an $L$-polarization).
\end{definition}

We will describe moduli spaces for marked $L$-polarized rational elliptic surfaces of type $d$ by extending the results of \cite[Section 6]{msacc}. This extension is largely fairly straightforward; we provide enough details to make precise statements of the main results and leave the remainder to the reader.

Let $\calM_{d,L}$ denote the moduli stack of families of marked $L$-polarized rational elliptic surfaces of type $d$ (here we work over the analytic category, so $\calM_{d,L}$ is a stack over the category of analytic spaces). More precisely, for an analytic space $S$, the objects of the category $\calM_{d,L}(S)$ are morphisms
\[\pi\colon (\calY, \calD = \calD_1 + \cdots + \calD_d) \longrightarrow S\]
together with an isomorphism
\[\mu\colon R^2\pi_*\bZ \stackrel{\sim}{\longrightarrow} \mathrm{I}_{1,9} \times S \]
so that the following conditions hold
\begin{enumerate}
\item the morphism $\calY \to S$ is a flat family of surfaces;
\item the analytic space $\calD_i$ is a divisor on $\calY/S$ for each $1 \leq i \leq d$; and
\item each closed fibre $(\calY_s, \calD_s, \mu_s)$ is a marked $L$-polarized rational elliptic surface of type $d$.
\end{enumerate}
The morphisms in the category from $(\calY,\calD)/S$ to $(\calY',\calD')/S'$ over a morphism $S \to S'$ are isomorphisms
\[(\calY,\calD) \stackrel{\sim}{\longrightarrow} (\calY',\calD') \times_{S'} S\]
over $S$ compatible with the markings.

For a generic $L$-polarized rational elliptic surface $(Y,D)$ of type $d$, let $K_{d,L}$ denote the subgroup of the automorphism group of $Y$ that acts trivially on $\Pic(Y)$. Then every object $(\calY,\calD)/S$ in $\calM_{d,L}(S)$ has a canonical subgroup $K_{d,L} \times S$ of its automorphism group, consisting of automorphisms which act trivially on the Picard groups of the fibres of $\calY$. Let $\tilde{\calM}_{d,L}$ denote the rigidification of $\calM_{d,L}$ along $K_{d,L}$, in the sense of \cite[Section 5]{tbac}; the objects of $\calM_{d,L}$ and $\tilde{\calM}_{d,L}$ coincide locally, but the automorphism group in $\tilde{\calM}_{d,L}$ is the quotient of the automorphism group in $\calM_{d,L}$ by $K_{d,L}$. 

The moduli stack $\calM_{d,0}$ is precisely the moduli stack of marked rational elliptic surfaces of type $d$ and $\tilde{\calM}_{d,0}$ is its rigidification. Results in \cite{msacc} show that the period map $\calM_{d,0} \to \Hom(F_d^{\perp},\bG_m)$ is well-defined and descends to a period map $\tilde{\calM}_{d,0} \to \Hom(F_d^{\perp},\bG_m)$. By the discussion above, these maps restrict to maps $\calM_{d,0} \to \calP_{d,0}$ and $\tilde{\calM}_{d,0} \to \calP_{d,0}$ respectively, where $\calP_{d,0} := \Hom(F_d^{\perp}/\langle f_0 \rangle, \bG_m)$.

Next we define the period map and period domain for more general marked $L$-polarized  rational elliptic surfaces of type $d$. To do this, we first need a small lemma.

\begin{lemma} \label{lem:periods} Let $\phi\colon F_d^{\perp} \longrightarrow \Pic^0(D) \cong \bG_m$ denote the period point of a marked $L$-polarized rational elliptic surface $(Y,D,\mu)$ of type $d$. Then $\phi(\alpha) = 1$ for any positive root $\alpha \in \Phi^+_L$.

Conversely, if $\alpha \in F_d^{\perp}$ is any class with $\alpha^2 = -2$ and $\phi(\alpha) = 1$, then either $\mu^{-1}(\alpha)$ or $\mu^{-1}(-\alpha)$ is contained in the effective cone in $\Pic(Y)$.
\end{lemma}

\begin{proof} Suppose that $\alpha \in\Phi^+_L$ is a positive root. By definition, $\mu^{-1}(\alpha) = \calO_Y(C)$, for some effective divisor $C \subset Y$ with $C.D = 0$. But then $\calO_Y(C)|_D \cong \calO_D$, so we have $\phi(\alpha) = 1$.

For the converse statement, suppose that  $\alpha \in F_d^{\perp}$ is any class with $\alpha^2 = -2$ and $\phi(\alpha) = 1$. Then $\mu^{-1}(\alpha) =: \calL$ is a line bundle on $Y$ with $\calL^2 = -2$ and $\calL|_{D} \cong \calO_D$. It follows immediately from \cite[Lemma 3.3]{msacc} that either $\calL$ or $\calL^{-1}$ is effective.
\end{proof}

From this lemma, we see that that the period point of a marked $L$-polarized rational elliptic surface $(Y,D,\mu)$ of type $d$ is completely determined by the action of $\phi$ on the quotient $F_d^{\perp}/\langle R_L, f_0 \rangle$. 

\begin{definition} The \emph{period domain} for marked $L$-polarized rational elliptic surfaces of type $d$ is the space 
\[\calP_{d,L} := \Hom(F_d^{\perp}/\langle R_L, f_0 \rangle,\bG_m).\]
There are natural \emph{period maps} $\calM_{d,L} \to \calP_{d,L}$ and $\tilde{\calM}_{d,L} \to \calP_{d,L}$, given by restriction of the usual (unpolarized) period map.
\end{definition}

Let $\calC_{d,L}^{++} \subset \mathrm{I}_{1,9}$ denote the subcone of $\calC_d^{++}$ consisting of those $\beta \in \calC^{++}_d$ satisfying 
\begin{enumerate}
\item $\beta.\delta_i \geq 0$ for the classes $\delta_i$ from Definitions \ref{def:marking} and \ref{def:marking'}; and
\item $\beta.\alpha \geq 0$ for all positive roots $\alpha \in \Phi_L^+$.
\end{enumerate}

Let $\Psi_{d,L} \subset F_d^{\perp}$ denote the set of roots $\alpha \in F_d^{\perp} \setminus R_L$ such that there exists a marked $L$-polarized rational elliptic surface $(Y,D,\mu)$ of type $d$ with $\mu^{-1}(\alpha)$ effective. $\Psi_{d,L}$ should be thought of as the possible ways to ``enhance'' $L$: it is the set of roots which could be added to $L$ and still give a valid polarizing lattice.

Now let $\Sigma$ denote the set of connected components of the complement
\[\calC_{d,L}^{++}\, \bigg\backslash\, \bigcup_{\alpha \in \Psi_{d,L}} \alpha^{\perp},\]
and let $Q = \calP_{d,L} \times \Sigma$. Define an \'{e}tale equivalence relation on $Q$ as follows: $(\phi,\sigma) \sim (\phi,\sigma')$ if and only if $\sigma$ and $\sigma'$ are contained in the same connected component of $\calC_{d,L}^{++} \setminus \bigcup_{\alpha \in \Psi_{\phi}} \alpha^{\perp}$, where
\[\Psi_{\phi} := \left\{\alpha \in \Psi_{d,L}\, \middle|\, \phi(\alpha) = 1\right\}.\]
By Lemma \ref{lem:periods}, we see that $\Psi_{\phi}$ is precisely the set of enhancements of $L$ that occur at the period point $\phi$; in other words, if $(Y,D,\mu)$ is a marked $L$-polarized rational elliptic surface of type $d$ with period point $\phi$, then $\Psi_{\phi}$ encodes the set of effective $(-2)$-classes in $\Pic(Y)$ which do not intersect any component of $D$ and do not lie in $\mu^{-1}(L)$.

Finally, let $\tilde{\calP}_{d,L} := Q / \sim$ denote the quotient of $Q$ by this equivalence relation. Note that we have a natural map $\tilde{\calP}_{d,L} \to \calP_{d,L}$ given by projection, which is an isomorphism over the analytic open set $U_{d,L} := \{\phi \in \calP_{d,L} \mid \Psi_{\phi} = \emptyset\}$.

Then we have the following theorem, which should be thought of as an analogue of \cite[Theorem 6.1]{msacc}.

\begin{theorem} \label{thm:ratellmoduli} There is an isomorphism $\tilde{\calM}_{d,L} \cong \tilde{\calP}_{d,L}$ which is compatible with the period map.
\end{theorem}
\begin{proof} The result follows from exactly the same argument used to prove \cite[Theorem 6.1]{msacc}. \end{proof}

In particular, we see that the period map $\tilde{\calM}_{d,L} \to \calP_{d,L}$ is an isomorphism over the open set $U_{d,L}$. This open set corresponds to the set of marked $L$-polarized rational elliptic surfaces of type $d$ that do not admit polarizations by any overlattice $L' \supset L$ with $R_{L'}$ strictly larger than $R_L$. 

Now let $L' \subset F_d^{\perp}$ be another negative definite lattice, with root lattice $R_{L'}$ and set of positive roots $\Phi_{L'}^+ \subset R_{L'}$. Suppose that $L \subset L'$ and $\Phi_L^+ \subset \Phi_{L'}^+ \subset \Phi_L^+ \cup \Psi_{d,L}$. Define $\Psi_{L'/L} := L' \cap  \Psi_{d,L} = \Phi_{L'}^+ \setminus \Phi_L^+$. Then there is a natural embedding of period domains $\calP_{d,L'} \hookrightarrow \calP_{d,L}$, whose image is the set $\{\phi \in \calP_{d,L} \mid \Psi_{L'/L} \subset \Psi_{\phi}\}$. 

From this and the discussion above, it is easy to see that the period domain $\calP_{d,L}$ admits a stratification 
\begin{equation}\label{eq:strata} \calP_{d,L} = \coprod_{L \subset L'} U_{d,L'},\end{equation}
where the disjoint union is taken over all negative definite overlattices $L \subset L' \subset F_d^{\perp}$ which may be obtained from $L$ by adjoining roots from $\Psi_{d,L}$. The strata may be identified with (open sets inside) moduli spaces of marked $L'$-polarized rational elliptic surfaces of type $d$, as above.

\section{Lattice polarized weak del Pezzo surfaces}\label{sec:dp}

Next we perform a similar analysis for weak del Pezzo surfaces. We begin by recalling the following definition.

\begin{definition} A \emph{weak del Pezzo surface} is a smooth rational surface $X$ which has nef and big anticanonical divisor $-K_X$.
\end{definition}

Results of Demazure \cite{sdp} and Coray and Tsfasman \cite[Proposition 0.4]{asdps} show that any weak del Pezzo surface is isomorphic to either $\bP^1 \times \bP^1$, the Hirzebruch surface $\bF_2$, or a blow-up of $\bP^2$ at $0 \leq n \leq 8$ points in almost general position (a set of points is in \emph{almost general position} if no stage of the blowing-up involves blowing up a point which lies on a $(-2)$-curve; in particular, infinitely near points are allowed as long as this condition is not violated).

In light of this, define a \emph{weak del Pezzo pair of degree $d$}, for $1 \leq d \leq 9$, to be a pair $(X,C)$ consisting of a weak del Pezzo surface $X$ obtained as a blow-up of $\bP^2$ in $(9-d)$ points in almost general position, along with a \emph{smooth} anticanonical divisor $C \in |-K_X|$. Note that $C^2 = d$.

However, this does not quite encompass all possible cases: we also need to account for $\bP^1 \times \bP^1$ and $\bF_2$. Define a \emph{weak del Pezzo pair of degree $8'$} to be a pair $(X,C)$, where $X = \bP^1 \times \bP^1$ or $\bF_2$ and $C \in |-K_X|$ is a smooth anticanonical divisor. In both of these cases $C^2 = 8$.

Note that the Picard lattice $\Pic(X)$ for a weak del Pezzo pair $(X,C)$ of degree $d$ is isometric to $\mathrm{I}_{1,9-d}$ (for $1 \leq d \leq 9$) and $\mathrm{II}_{1,1}$ when $d = 8'$. Denote the lattice $\mathrm{I}_{1,9-d}$ (resp. $\mathrm{II}_{1,1}$) by $\Lambda_d$ (resp. $\Lambda_{8'}$). Recall from Section \ref{sec:lattices} that the lattices $\Lambda_d$ contain a distinguished class $f_d$.

In the remainder of this paper, when no confusion is likely to result, we will simply refer to weak del Pezzo pairs of degree $d$, distinguished classes $f_d$, and lattices $\Lambda_d$, with the understanding that this includes the case $d = 8'$ unless otherwise specified. 

As in the previous section, we can define a marking on a weak del Pezzo surface of degree $d$.

\begin{definition} \label{defn:dPmarking} A \emph{marking} on a weak del Pezzo pair $(X,C)$ of degree $d$ is a choice of isometry $\mu\colon \Pic(X) \to \Lambda_d$ satisfying $\mu(C) = f_d$.
\end{definition}

\begin{remark} \label{rem:C++} To make reasonable statements about the moduli spaces of marked weak del Pezzo pairs, we suspect that this definition should also contain a condition analogous to the condition on $\calC^{++}_d$ in Definition \ref{def:marking}. This is needed to kill any automorphisms of $\Lambda_d$ that are not realized by deformations of $(X,C)$.
\end{remark}

\subsection{Lattice polarizations and the K\"{a}hler cone} \label{sec:dPpol} We next define the concept of a \emph{lattice polarization}. Let $f_d^{\perp}$ denote the orthogonal complement of $f_d$ in $\Lambda_d$. Note that $f_d^{\perp}$ is negative definite of rank $9 - d$, with intersection form given by Table \ref{tab:latticetype}. Comparing this table to Table \ref{tab:ellipticlatticetype}, we discover that $f_d^{\perp} \cong F_d^{\perp}/\langle f_0 \rangle$ for all $d$; this is the first manifestation of  mirror symmetry in this setting

Let $L \subset f_d^{\perp}$ denote a primitive sublattice and let $R_L$ denote the sublattice of $L$ generated by roots $\alpha \in L$. Both $L$ and $R_L$ are negative definite so, in particular, $R_L$ is a root lattice.

\begin{table}
\begin{tabular}{|c||c|c|c|c|c|c|c|c|c|c|}
\hline
$d$ & $1$ & $2$ & $3$ & $4$ & $5$ & $6$ & $7$ & $8$ & $8'$ & $9$ \\
\hline
\rule{0pt}{3ex} $f_d^{\perp}$ & $E_8$ & $E_7$ & $E_6$ & $D_5$ & $A_4$ & $A_2 + A_1$ & $\left(\begin{smallmatrix} -2 & 1 \\ 1 & -4 \end{smallmatrix}\right)$ & $( -8 )$ & $A_1$ & $\{0\}$ \\
\hline
\end{tabular}
\caption{Lattices $f_d^{\perp}$}
\label{tab:latticetype}
\end{table}

\begin{definition}   An \emph{$L$-polarization} of a weak del Pezzo pair $(X,C)$ of type $d$ is a primitive embedding $\nu\colon L \hookrightarrow \Pic(X)$ such that
\begin{itemize}
\item $\nu(\beta).C = 0$ for all $\beta \in L$, and
\item there exists a set of positive roots $\Phi^+$ in $R_L$ so that $\nu(\Phi^+)$ is contained in the effective cone in $\Pic(X)$.
\end{itemize}
\end{definition}

\begin{remark} As in the previous section (see Remark \ref{rem:ellipticroot}), we see that the essential part of this definition really only depends upon the root lattice $R_L$. Indeed, we would not really lose anything by always taking $L$ to be a root sublattice of $f_d^{\perp}$. However, as before, the definition above has better mirror symmetric properties.
\end{remark}

As before, we may also define the notion of a marked $L$-polarization. For the remainder of this section, assume that we have fixed a primitive embedding $L \hookrightarrow f_d^{\perp} \subset \Lambda_d$ and chosen a set of positive roots $\Phi_L^+ \subset R_L$.

\begin{definition} A \emph{marked $L$-polarization} on a weak del Pezzo pair $(X,C)$ of type $d$ is a choice of marking $\mu \colon \Pic(X) \to \Lambda_d$ such that the preimage $\mu^{-1}(\Phi_L^+)$ is contained in the effective cone in $\Pic(X)$ (i.e. so that $\mu^{-1}|_L$ is an $L$-polarization).
\end{definition}

Next we look at the K\"{a}hler cone of a lattice polarized weak del Pezzo pair of degree $d$. Let $(X,C,\mu)$ denote a marked $L$-polarized weak del Pezzo pair of degree $d$. The K\"{a}hler cone of $X$ is equal to its ample cone, since the N\'{e}ron-Severi lattice $\NS(X)$ is equal to $H^2(X,\bZ)$, and its closure is the nef cone $\Nef(X)$. The nef cone $\Nef(X)$ is dual to the cone of effective curves on $X$ which, by \cite[Proposition 6.2]{dhmq}, is rational polyhedral and generated by a finite set of rays. We would like to see what information the lattice polarization gives us about the structure of these cones. 

\begin{proposition} Let $(X,C,\mu)$ be a marked $L$-polarized weak del Pezzo pair of degree $d$. If $\alpha \in \Phi_L^+$ is a simple root, then $(\mu^{-1}(\alpha))^{\perp}$ is a codimension $1$ face of $\Nef(X)$ which contains the ray generated by $[C] = [-K_X]$.
\end{proposition}
\begin{proof} Let $\alpha \in \Phi_L^+$ be a simple root. Adjunction shows that $\mu^{-1}(\alpha)$ is the class of an irreducible rational $(-2)$-curve in $X$. It then follows from the proof of \cite[Proposition 6.2]{dhmq} that $\mu^{-1}(\alpha)$ is an extremal ray in the intersection of the effective cone of $X$ with the hyperplane $[-K_X]^{\perp}$. Since $\Nef(X)$ is the dual of the effective cone of $X$, we see that $(\mu^{-1}(\alpha))^{\perp}$ is a codimension $1$ face of $\Nef(X)$ which contains the ray generated by $[-K_X]$.
\end{proof}

With an additional assumption, this enables us to completely describe the structure of $\Nef(X)$ in a neighbourhood of the ray generated by $[C] = [-K_X]$. We say that a marked $L$-polarized weak del Pezzo pair $(X,C,\mu)$ is \emph{generic} if $\mu^{-1}(\Phi_L^+)$ contains the classes of all irreducible rational $(-2)$-curves on $X$.

\begin{proposition} \label{prop:dPKahler} Let $(X,C,\mu)$ be a generic marked $L$-polarized weak del Pezzo pair of degree $d$. Then in a neighbourhood of the ray generated by $[C]=[-K_X]$, the codimension $1$ faces of $\Nef(X)$ are given by the hyperplanes $(\mu^{-1}(\alpha))^{\perp}$, where $\alpha \in \Phi_L^+$ is a simple root.
\end{proposition}
\begin{proof} It remains to show that every codimension $1$ face in a neighbourhood of the ray generated by $[-K_X]$ arises from a simple root in $\Phi_L^+$. The proof of \cite[Proposition 6.2]{dhmq} shows that such faces are orthogonal to the classes of irreducible rational $(-2)$-curves in $\NS(X)$ and, by genericity, all such classes are contained in $\mu^{-1}(\Phi_L^+)$. Irreducibility implies that such classes must come from simple roots.
\end{proof}

\section{Mirror symmetry for lattice polarized weak del Pezzo pairs}\label{sec:mirror}

Now we come to the main section of this paper. We claim that mirror symmetry for del Pezzo surfaces can be extended to the lattice polarized setting.

First we review classical mirror symmetry for del Pezzo surfaces, as given by the Fano-LG correspondence. Using techniques of homological mirror symmetry, Auroux, Katzarkov, and Orlov \cite{msdpsvccs} have proposed that the mirror to a weak del Pezzo pair of degree $d$ should be given by an elliptic fibration over $\bC$ with $12-d$ singular fibres, each of type $I_1$, which admits a compactification to a rational elliptic surface of type $d$. In our framework, this should be thought of as analogous to the statement ``the topological mirror to a K3 surface is a K3 surface''. As in the K3 case, we claim that we can gain a significantly richer understanding by studying the effect of mirror symmetry on lattice polarizations.

The mirror construction for lattice polarized weak del Pezzo pairs proceeds as follows. Let $(X,C)$ be a weak del Pezzo pair of degree $d$ that is polarized by a lattice $L \subset f_d^{\perp}$. Let $\check{L}:=(L)^{\perp}_{f_d^{\perp}}$ denote the orthogonal complement of $L$ in $f_d^{\perp} \cong F_d^{\perp}/\langle f_0 \rangle$; note that $\check{L}$ is negative definite. 

Now, note that the lattice $F_d^{\perp}$ (as given in Table \ref{tab:ellipticlatticetype}) contains the lattice $f_d^{\perp}$ (as given in Table \ref{tab:latticetype}) as a sublattice. So we may choose an embedding $f_d^{\perp} \hookrightarrow F_d^{\perp}$, thereby identifying $\check{L}$ with a negative definite sublattice of $F_d^{\perp}$. Then we have the following mirror conjecture.

\begin{conjecture}\label{conj:mirror} The mirror to an $L$-polarized weak del Pezzo pair $(X,C)$ of degree $d$ is given by the open set $Y \setminus D$ in an $\check{L}$-polarized rational elliptic surface $(Y,D)$ of type $d$.
\end{conjecture}

\begin{remark} The choice of the embedding $f_d^{\perp} \hookrightarrow F_d^{\perp}$ in the formulation of this conjecture should be thought of as analogous to the choice of a maximally unipotent monodromy point in the usual formulation of mirror symmetry; see Remark \ref{rem:choice}.
\end{remark}

\begin{remark} At this point, we briefly remark on what we would expect to happen if we had defined our weak del Pezzo pairs $(X,C)$ to have $C$ singular, rather than smooth. In this case, the philosophy of mirror symmetry suggests that for each nodal singularity appearing in $C$, one should remove a section from the mirror rational elliptic surface $(Y,D)$, making it into a fibration by punctured elliptic curves. For instance, in the degree $9$ case the pair $(X,C)$ consists of a smooth cubic in $\bP^2$, which may degenerate to contain one, two, or three nodes. On the mirror side, these degenerations correspond to removing one, two, or three sections, respectively, from a rational elliptic surface of type $9$; note that such a rational elliptic surface has Mordell-Weil group $\bZ/3\bZ$ \cite{mwlres}. This idea will reappear in Section \ref{sec:gs}, where we will use it to simplify the computation of Gross-Siebert mirrors.
\end{remark}

\subsection{Complex and K\"{a}hler Moduli}\label{sec:complexkahler}

As a first piece of evidence for this conjecture, we discuss the correspondence between the complex moduli of a lattice polarized rational elliptic surface of type $d$, as discussed in Section \ref{sec:ratellpol}, and the K\"{a}hler cone of a lattice polarized weak del Pezzo pair, as discussed in Section \ref{sec:dPpol}.

Begin by letting $(X,C)$ be a weak del Pezzo pair of degree $d$. Since $b_2(X) = 10-d$, the K\"{a}hler cone of $X$ must have dimension $10 - d$. Thus, in order to have any hope of making a mirror symmetric statement relating the boundary components of the K\"{a}hler cone of $X$ to the strata in the period domain of its mirror, we will need to choose our polarizing lattice $L$ so that the period domain for $\check{L}$-polarized rational elliptic surfaces of type $d$ has dimension $9 - d$. This happens if and only if the root sublattice $R_{\check{L}}$ is trivial.

To ensure this is that case, choose $L$ to equal the lattice $f_d^{\perp}$ from Table \ref{tab:latticetype} and assume that $(X,C)$ is a generic marked $L$-polarized weak del Pezzo pair of degree $d$. In this setting we have a description of the boundary of the nef cone of $X$ (which, we recall, is the closure of the K\"{a}hler cone of $X$) in a neighbourhood of the ray generated by $[C] = [-K_X]$, given by Proposition \ref{prop:dPKahler}; recall that codimension $1$ faces correspond to simple roots in $\Phi_L^+ \subset L$.

Now we proceed to the mirror. We have $\check{L} = 0$, as required, and from Equation \eqref{eq:strata} we obtain a stratification of the period domain $\calP_{d,0}$. We may describe some of the strata in this stratification as follows. Recall that in the process of defining a mirror, we chose an embedding $f_d^{\perp} \hookrightarrow F_d^{\perp}$. This embedding allows us to identify $L$ as a corank $1$ sublattice of $F_d^{\perp}$, which gives rise to a $1$-dimensional stratum $U_{d,L}$ in $\calP_{d,0}$. In a neighbourhood of this stratum, the strata in $\calP_{d,0}$ are given by lattices $L' \subset L$ which are generated by roots in $\Phi_L^+ \subset L$.

We thus obtain a correspondence between
\begin{itemize}
\item boundary components of the nef cone of $X$ in a neighbourhood of the ray generated by $[C] = [-K_X]$, and
\item strata of  $\calP_{d,0}$ in a neighbourhood of the $1$-dimensional stratum $U_{d,L}$, corresponding to sublattices $L' \subset L$ which are generated by simple roots in $\Phi_L^+ \subset L$.
\end{itemize}
This should be thought of as a mirror correspondence; we will make some attempt to formalize it in Section \ref{sec:GSconjecture} using the Gross-Siebert program.
\medskip

It is interesting to consider whether there is a similar correspondence relating the nef cone of a rational elliptic surface of type $d$ to a moduli space for weak del Pezzo pairs of degree $d$. This question is significantly more difficult, for several reasons.

The nef cone of a rational elliptic surface $Y$ has been widely studied. Borcea \cite[Section 6]{dhmq} showed that the extremal rays generating the effective cone (which is dual to the nef cone) of $Y$ may accumulate to the ray generated by $[-K_Y]$. If this happens, the nef cone of $Y$ will not be rational polyhedral in a neighbourhood of the ray generated by $[-K_Y]$; by work of Totaro \cite[Theorem 8.2]{h14pffccc}, this occurs if and only if the Mordell-Weil group of $Y$ is infinite. The Mordell-Weil groups of all rational elliptic surfaces were computed explicitly by Oguiso and Shioda in \cite{mwlres}; they may be infinite even in cases admitting polarizations by one of the lattices from Table \ref{tab:latticetype}. Due to this, it seems quite unlikely that a direct analogue of Proposition \ref{prop:dPKahler} will hold for rational elliptic surfaces of type $d$. 

On the other side, moduli spaces of del Pezzo surfaces have also been widely studied, notably by Colombo, van Geeman, and Looijenga \cite{dpmrs}, and by Hacking, Keel, and Tevelev \cite{sptlccmsdps}. However, for our purposes the correct objects to study should be \emph{weak del Pezzo pairs} $(X,C)$, where $X$ is a weak del Pezzo surface and $C$ is a smooth anticanonical divisor on $X$. The moduli of such pairs has been studied by Friedman \cite{mhsov} and McMullen \cite{dbpp} but, to our knowledge, a complete description in terms of periods (analogous to that in \cite{rsac,msacc}) has not been worked out.

That said, we do expect the moduli spaces of weak del Pezzo pairs to admit such a description. Fix a smooth elliptic curve $C$ and let $(X,C,\mu)$ denote a marked weak del Pezzo pair of degree $d$ whose anticanonical curve is isomorphic to $C$. Define the \emph{period point} $\phi \in \Hom(f_d^{\perp},C)$ associated to $(X,C,\mu)$ by
\begin{align*}
\phi\colon f_d^{\perp} &\longrightarrow \Pic^0(C) \cong C \\
\alpha &\longmapsto \mu^{-1}(\alpha)|_C.
\end{align*}
McMullen \cite[Corollary 4.4 and Theorem 6.4]{dbpp} has shown that $(X,C,\mu)$ is determined up to isomorphism by its period point $\phi$, and surjectivity of the period map follows easily from his results. However, in order to prove a global Torelli theorem for marked weak del Pezzo pairs, and hence relate their moduli space to the period domain $\Hom(f_d^{\perp},C)$, one needs better control over the action of the automorphism group of $\Lambda_d$ on the markings. We expect that this will require additional conditions in the definition of a marking; see Remark \ref{rem:C++}.

Based on these partial results, and the results of Section \ref{sec:ratellpol}, we make the following conjecture about the moduli spaces of marked lattice polarized weak del Pezzo pairs.

\begin{conjecture} \label{conj:dPmoduli} Fix a smooth elliptic curve $C$. There is a period map from the moduli space of marked $L$-polarized weak del Pezzo pairs $(X,C,\mu)$ of degree $d$, with anticanonical curve $C$, to $\Hom(f_d^{\perp}/R_L,C)$, and this period map is an isomorphism over a dense open set $U \subset \Hom(f_d^{\perp}/R_L,C)$.
\end{conjecture}

\begin{remark} \label{rem:ako} We briefly remark  upon the relationship between the ideas in this section and the results of \cite{msdpsvccs}. Note that, in contrast to the discussion above, the authors of \cite{msdpsvccs} are primarily concerned with comparing the complex structure on $(X,C)$ with the symplectic structure on $Y \setminus D$. This allows them to largely ignore the complex structure of $Y \setminus D$; in particular, they make the assumption that the critical points of the elliptic fibration on $Y \setminus D$ are all isolated and non-degenerate, so all singular fibres have Kodaira type $\mathrm{I}_1$. This assumption is required by Seidel's \cite{vcm} definition of the derived category of Lagrangian vanishing cycles, which is needed in homological mirror symmetry.

We claim that the formulation of mirror symmetry thus obtained agrees with our conjecture \emph{up to deformation}. In our setting, different choices of lattice polarization on $(X,C)$ will lead to different configurations of singular fibres on $Y \setminus D$, but all configurations may be deformed to the mirror as described by \cite{msdpsvccs}. It would be interesting to see whether homological mirror symmetry statements, such as those in \cite{msdpsvccs}, could also be made in cases where $Y \setminus D$ has more severe singular fibres.
\end{remark}

\section{Compatibility with Dolgachev-Nikulin-Pinkham mirror symmetry for K3 surfaces} \label{sect:K3}

In the remainder of this paper, we discuss the compatibility of the mirror construction from Section \ref{sec:mirror} with various other forms of mirror symmetry. We begin by analysing how it fits with Dolgachev-Nikulin-Pinkham mirror symmetry for lattice polarized K3 surfaces \cite{sedeask3,fagkk3s,iqfaag,mslpk3s}. The basic idea behind this correspondence is set out in \cite[Section 4]{mstdfcym}; this section provides a much more detailed picture.

We begin with some setup. We first describe how to endow a family of K3 surfaces with a lattice polarization, following \cite{flpk3sm}. Suppose that $\pi\colon \calV^* \to \Delta^*$ is a smooth family of K3 surfaces over the punctured unit disc $\Delta^* \subset \bC$. Let $V_p$ denote the fibre of $\calV$ over $p \in \Delta^*$ and let $\NS(V_p)$ denote its N\'{e}ron-Severi lattice. Finally, let $M \subset \Lambda_{\mathrm{K3}}$ be a primitive sublattice of the \emph{K3 lattice} $\Lambda_{\mathrm{K3}} := H^{\oplus 3} \oplus E_8^{\oplus 2}$, where $H = \mathrm{II}_{1,1}$ denotes the hyperbolic plane and $E_8$ is the root lattice corresponding to the Dynkin diagram $E_8$.

\begin{definition}\cite[Definition 2.1]{flpk3sm} \label{def:Lpol} $\pi\colon\calV^*\to \Delta^*$ is an \emph{$M$-polarized family of K3 surfaces} if
\begin{itemize}
\item there is a trivial local subsystem $\calM$ of $R^2\pi_*\bZ$ so that, for each $p \in \Delta^*$, the fibre $\calM_p \subset H^2(V_p,\bZ)$ of $\calM$ over $p$ is a primitive sublattice of $\NS(V_p)$ that is isomorphic to $M$, and
\item there is a line bundle $\calA$ on $\calV^*$ whose restriction $\calA_p$ to any fibre $V_p$ is ample with first Chern class $c_1(\calA_p)$ contained in $\calM_p$ and primitive in $\NS(V_p)$.
\end{itemize}
\end{definition}

Now, suppose that $\pi \colon \calV \to \Delta$ is a type II degeneration of K3 surfaces, where $\Delta \subset \bC$ denotes the unit disc. Recall that this means that
\begin{itemize}
\item $\calV$ is a nonsingular threefold with trivial canonical bundle;
\item the restriction of $\pi$ to $\Delta^*$ is a smooth morphism whose fibres are K3 surfaces; and
\item the central fibre $V_0$ of $\calV$ is a chain of smooth surfaces $X_1 \cup_{E_1} X_2 \cup_{E_2} \cdots \cup_{E_k} X_k$, where $X_1$ and $X_k$ are rational, $X_2,\ldots,X_{k-1}$ are elliptic ruled, and $X_i$ and $X_{i+1}$ meet normally along a smooth elliptic curve $E_i$. 
\end{itemize}

We will further assume that the restriction of $\pi$ to $\Delta^*$ is an $M$-polarized family of K3 surfaces, in the sense of Definition \ref{def:Lpol}, for some $M \subset \Lambda_{\mathrm{K3}}$. Moreover, so that we can make a comparison with the notion of mirror symmetry defined in Section \ref{sec:mirror}, we will assume that $X_1$ is a weak del Pezzo surface. By adjunction and the condition that $\omega_{\calV} \cong \calO_{\calV}$, we see that $E_1 \in |-K_{X_1}|$, so $(X_1,E_1)$ is a weak del Pezzo pair. Let $d$ denote its degree.  

\begin{example} A simple example of the type of degeneration discussed above is as follows. Consider the family
\[\calV := \{z^2 = (f_3(x_1,x_2,x_3))^2 + tg_6(x_1,x_2,x_3)\} \subset \WP(1,1,1,3)[x_1,x_2,x_3,z]\]
over the disc $\Delta$ with coordinate $t$, where $f_3$ and $g_6$ denote a generic cubic and sextic respectively. $\calV$ is singular at the $18$ points $f_3 = g_6 = z = t = 0$, but we may perform a simultaneous small resolution by blowing up the ideal $\langle t,z-f_3 \rangle$ in $\calV$ (this ideal defines a Weil divisor in $\calV$ that is not $\bQ$-Cartier). The result is a type II degeneration of K3 surfaces that is polarized by the lattice $\langle 2\rangle$. The central fibre has two components $X_1 \cup_{E_1} X_2$, and $(X_1,E_1)$ is a weak del Pezzo pair of degree $9$ (in this case, $X_2$ is a copy of $\bP^2$ blown up in $18$ points, so is not weak del Pezzo).

Several more complicated examples are described in \cite[Section 4]{mstdfcym}.
\end{example}

Now we begin to discuss mirror symmetry. Our type II degeneration corresponds to a $1$-dimensional cusp $\calC_1$ in the Baily-Borel compactification $\overline{\calD}_M$ of the period domain $\calD_M$ of $M$-polarized K3 surfaces. By the discussion in \cite[Section 2.1]{cmsak3s}, such cusps are in bijective correspondence with rank two isotropic sublattices of $M^{\perp}$. Denote the rank $2$ isotropic sublattice corresponding to $\calC_1$ by $I$.

In order for us to apply Dolgachev-Nikulin-Pinkham mirror symmetry, we must assume that $\calC_1$ contains a $0$-dimensional cusp (Type III point) $\calC_0$. Such cusps correspond to primitive isotropic vectors $e \in M^{\perp}$ (satisfying certain conditions), and the condition that $\calC_0 \subset \calC_1$ is equivalent to requiring that $e \in I$.

Under these conditions, the Dolgachev-Nikulin-Pinkham mirror to a general fibre of $\calV$ is a K3 surface $W$ polarized by the lattice
\[\check{M} := (\bZ e)^{\perp}_{M^{\perp}}\, / \, \bZ e.\]
Note that, by \cite[Proposition 4.1]{mstdfcym}, the fact that $e$ lies in the rank two isotropic lattice $I$ is equivalent to the existence of a primitive isotropic vector $f \in \check{M}$, which gives rise to an elliptic fibration on $W$. The vectors $e$ and $f$ span $I$.

\subsection{The relationship with mirror symmetry for weak del Pezzo pairs}

We begin by relating the lattice polarization on $\calV$ to one on the weak del Pezzo pair $(X_1,E_1)$. To do this, we use the Clemens-Schmid exact sequence of mixed Hodge structures. This exact sequence gives rise to an exact sequence of weight graded pieces
\begin{equation} \label{eq:cssequence} 0 \to \Gr_0(H^0_{\lim}) \to \Gr_{-4}(H_4(V_0)) \stackrel{\varphi}{\to} \Gr_2(H^2(V_0)) \stackrel{i^*}{\to} \Gr_2(H^2_{\lim}) \to 0,\end{equation}
where $H^k_{\lim}$ denotes the limiting mixed Hodge structure on the $k$th cohomology of $V_p$, for a general $p \in \Delta^*$, and $i^*$ denotes the natural map $H^2(V_0) \cong H^2(\calV) \to H^2(V_p)$ given by pulling-back by the inclusion $i\colon V_p \to \calV$.

We begin by examining these graded pieces in more detail. By \cite[Section 4]{csesa}, the weight filtration on $H^2_{\lim}$ is given by
\[\{0\} \subset I_{\bQ} \subset I^{\perp}_{\bQ} \subset H^2_{\lim} \cong H^2(V_p,\bQ),\]
where $I_{\bQ} := I \otimes \bQ$. Thus $\Gr_2(H^2_{\lim}) = I^{\perp}_{\bQ}/I_{\bQ}$. By construction, there is a natural injective map $M \hookrightarrow \Gr_2(H^2_{\lim})$.

 The weight filtration on $H^2(V_0)$ is given by Deligne's mixed Hodge structure on a normal crossing variety. We compute that
\[\Gr_2(H^2(V_0)) = \ker\left(\delta\colon \bigoplus_{i=1}^k H^2(X_i,\bQ) \to \bigoplus_{j=1}^{k-1} H^2(E_j,\bQ)\right),\]
where $\delta = (\delta_1,\ldots,\delta_{k-1})$ is given by the alternating sums $\delta_j:= \sum_{i=1}^k (-1)^i\delta_{ij}$ of the restriction maps $\delta_{ij}\colon H^2(X_i,\bQ) \to H^2(E_j,\bQ)$. This kernel may be identified with those classes $(D_1,\ldots,D_k) \in  \bigoplus_{i=1}^k H^2(X_i,\bQ)$ which satisfy the gluing condition $D_i|_{E_i} - D_{i+1}|_{E_i} = 0$ for all $i \in \{1,\ldots,k-1\}$.

\begin{lemma} \label{lemma:injection} The Clemens-Schmid exact sequence \eqref{eq:cssequence} induces an injective map $j \colon [E_1]^{\perp} \to \Gr_2(H^2_{\lim})$, where $[E_1]^{\perp}$ denotes  the orthogonal complement of $E_1$ in $H^2(X_1,\bZ) \cong \Lambda_d$.
\end{lemma}
\begin{proof} By the discussion above, $[E_1]^{\perp}$ is equal to the intersection of the weight graded piece $\Gr_2(H^2(V_0))$ with $H^2(X_1,\bZ)$. We claim that the restriction $j := i^*|_{[E_1]^{\perp}}$ is injective.

To prove this, we look at the weight filtration on $H_4(V_0)$. This weight filtration is Poincar\'{e} dual to the weight filtration on $H^4(V_0)$ arising from Deligne's mixed Hodge structure. It has $\Gr_{-4}(H_4(V_0)) = H_4(V_0)$, generated by the classes $[X_i]$. The image of the map $\varphi$  is those classes supported on the curves $E_i$; for such classes the gluing condition becomes the condition that the sum of the coefficients of $E_i|_{X_i}$ in $H^2(X_i,\bQ)$ and $E_i|_{X_{i+1}}$ in $H^2(X_{i+1},\bQ)$ is zero, along with the classical \emph{triple point formula} $(E_i|_{X_i})^2 + (E_i|_{X_{i+1}})^2 = 0$ (which holds for any type II degeneration). It is clear that the intersection of this image with $[E_1]^{\perp}$ is trivial, so $j$ is injective by exactness.
\end{proof}

\begin{remark}\label{rem:choice} The injection $j\colon [E_1]^{\perp} \hookrightarrow \Gr_2(H^2_{\lim})$ defined by this lemma induces an embedding of $[E_1]^{\perp}$ into the orthogonal complement $(\bQ f)^{\perp}/\bQ f$ of the class $f$ in $H^2_{\lim}$. This should be thought of as analogous to the embedding $f_d^{\perp} \hookrightarrow F_d^{\perp}$ in the formulation of mirror symmetry from Section \ref{sec:mirror}. 

The definition of this embedding depends upon the class $f \in I$, which in turn depends upon the choice of $0$-dimensional cusp $\calC_0$. Thus we see that the choice of embedding $f_d^{\perp} \hookrightarrow F_d^{\perp}$ corresponds to the choice of a maximally unipotent monodromy (Type III) point $\calC_0$ in the Dolgachev-Nikulin-Pinkham formulation of mirror symmetry.
\end{remark}

Now let $L$ denote the preimage of $M \otimes \bQ$ under $j$. Then $L$ is a primitive sublattice of $[E_1]^{\perp} \subset H^2(X_1,\bZ) \cong \Pic(X_1)$.

\begin{proposition} The weak del Pezzo pair $(X_1,E_1)$ is $L$-polarized.
\end{proposition}
\begin{proof} To show that $(X_1,E_1)$ is $L$-polarized, it remains to show that $R_L$ contains a set of positive roots $\Phi^+$ which lies in the effective cone in $\Pic(X_1)$.

To show this, it suffices to show that if $\alpha$ is any root in $L$, then either $\alpha$ or $-\alpha$ is effective in $H^2(X_1,\bZ)$ (the condition on sums of positive roots then follows immediately from the fact that the sum of two effective classes is effective). So let $\alpha \in L$ be a root. Then $\alpha' := j(\alpha)$ is an integral $(-2)$-class, so defines a root in $M$. As $\calV^* \to \Delta^*$ is an $M$-polarized family, $\alpha'$ thus defines a $(-2)$-class in $H^2(V_p,\bZ)$, for general $p \in \Delta^*$, that is invariant under monodromy.

Since $V_p$ is a K3 surface, a standard result \cite[Proposition VIII.3.7]{bpv} tells us that either $\alpha'$ or $-\alpha'$ is the class of an effective divisor on $V_p$. This class sweeps out a divisor on $\calV^*$, which extends to a divisor on $\calV$. The intersection of this divisor with $X_1$ is an effective divisor in the class $\alpha$ or $-\alpha$. Thus either $\alpha$ or $-\alpha$ is effective. \end{proof}

With this in place, we make the following definition.

\begin{definition} The lattice polarizations $L$ on $(X_1,E_1)$ and $M$ on $\calV^* \to \Delta^*$ are said to be \emph{compatible} if $j(L^\perp_{[E_1]^{\perp}}) \subset M^{\perp}$, where $L^\perp_{[E_1]^{\perp}}$ denotes the orthogonal complement of $L$ in ${[E_1]^{\perp}}$ and the orthogonal complement of $M$ is taken in the K3 lattice $\Lambda_{\mathrm{K3}}$.
\end{definition}

Throughout the remainder of this section we will always assume that the lattice polarizations $L$ and $M$ are compatible. The next proposition justifies this assumption by showing that, for any lattice polarized weak del Pezzo pair $(X,C)$, we may always find a compatible lattice polarized type II degeneration which has $(X,C)$ as a component of its central fibre.

\begin{proposition} \label{prop:extension} Let $(X,C)$ be an $L$-polarized del Pezzo pair of degree $d$, for some lattice $L$. Then there exists a Type II degeneration of K3 surfaces $\pi \colon \calV \to \Delta$ such that $(X,C)$ is a component of the central fibre of $\calV$ and the restriction of $\calV$ to $\Delta^*$ is an $M$-polarized family of K3 surfaces, for some lattice $M$ that is compatible with $L$.
\end{proposition}

\begin{proof} For clarity of notation, set $(X_1,E_1) := (X,C)$. We begin by finding a second pair $(X_2,E_1)$, so that $X_1 \cup_{E_1} X_2$ forms the central fibre of a Type II degeneration.

Choose a rational surface $X_2$ with a smooth anticanonical divisor isomorphic to $E_1$, so that $\calN_{E_1/X_1} \otimes \calN_{E_1/X_2} = \calO_{E_1}$ (this is Friedman's \cite{gsvnc} \emph{$d$-semistability} condition). One can construct such a pair $(X_2,E_1)$ as follows: begin with an embedding $E_1 \subset \bP^2$ (which is unique up to isomorphism), then choose $(18-d)$ points $P_1,\ldots,P_{18-d}$ on $E_1$ and blow them up to obtain $X_2$. The $d$-semistability criterion imposes a single condition on the $P_i$, so there is a $(17-d)$-dimensional family of such pairs $(X_2,E_1)$.

Next, choose a nef and big divisor $H$ on $X_2$ such that $\calO_{X_2}(H)|_{E_1} \cong \calO_{X_1}(E_1)|_{E_1}$. For the explicitly constructed pair above, one may take $H$ to be three times the pull-back of the hyperplane section from $\bP^2$ minus $d$ of  the exceptional curves; for this $H$ the condition $\calO_{X_2}(H)|_{E_1} \cong \calO_{X_1}(E_1)|_{E_1}$ imposes a second condition on the positions of the $P_i$,  so we have a $(16-d)$-dimensional family of suitable choices of $(X_2,E_1)$ and $H$.

Now glue the surfaces $X_1$ and $X_2$ along the curve $E_1$ to obtain a normal crossing surface $V_0 := X_1 \cup_{E_1} X_2$. By construction, $V_0$ is a polarized stable K3 surface of Type II \cite[Definition 3.1]{npgttk3s}, with polarizing class $h := (E_1,H) \in H^2(X_1,\bZ) \oplus H^2(X_2,\bZ)$. Define $M \subset H^2(X_1,\bZ) \oplus H^2(X_2,\bZ)$ to be the smallest primitive sublattice containing $L$ and $h$. By construction we have $M \subset \Gr_2(H^2(V_0))$, where we have equipped $H^2(V_0)$ with Deligne's mixed Hodge structure on a normal crossing variety.

Now, a mild modification to the proof of \cite[Proposition 4.3]{npgttk3s} shows that $V_0$ is the central fibre in a Type II degeneration of K3 surfaces $\pi \colon \calV \to \Delta$, such that the restriction of $\calV$ to $\Delta^*$ is an $M$-polarized family of K3 surfaces. It thus only remains to prove that $L$ and $M$ are compatible; this follows from the fact that $L^\perp_{[E_1]^{\perp}}$ is orthogonal to $M$.
\end{proof}

Now we proceed to the mirror. The next lemma allows us to relate the mirror lattices $\check{L}$ and $\check{M}$.

\begin{proposition} \label{prop:mirrorfibres} Assume that the lattice polarizations $L$ on $(X_1,E_1)$ and $M$ on $\calV^* \to \Delta^*$ are compatible. Then the mirror lattice $\check{L}$ is a sublattice of the orthogonal complement $(\bZ f)_{\check{M}}^{\perp}\,/ \, \bZ f$ of the vector $f$ in $\check{M}$.
\end{proposition}

\begin{proof} By definition, $j$ defines an injection $\check{L} \hookrightarrow M^{\perp}$. To complete the proof, we have to show that $\check{L}$ is in the orthogonal complement of both $e$ and $f$; i.e. that $\check{L} \subset I^{\perp}/I$. But this follows immediately from the fact that the image of $j$ lies in $\Gr_2(H^2_{\lim})$.
\end{proof}

Now we put all the pieces together. Let $W$ be a K3 surface that is Dolgachev-Nikulin-Pinkham mirror to a general fibre $V_p$ of $\calV$. Then there exists a lattice polarization $\nu\colon \check{M} \hookrightarrow \Pic(W)$. By the previous proposition, we may identify $\check{L}$ with a sublattice of $\check{M}$.

Let $R_{\check{L}}$ denote the sublattice of $\check{L}$ generated by roots and let $\alpha, \alpha' \in R_{\check{L}}$ be any two roots. It follows from standard results on K3 surfaces \cite[Proposition VIII.3.7]{bpv} that either $\nu(\alpha)$ or $\nu(-\alpha)$ is the class of an effective divisor on $W$, and if $\nu(\alpha)$ and $\nu(\alpha')$ are both classes of effective divisors, then so is $\nu(\alpha + \alpha')$. Thus there is a set of positive roots $\Phi^+ \subset R_{\check{L}}$ such that $\nu(\Phi^+)$ is contained in the effective cone in $\Pic(W)$.  

In fact, we can say more. Recall from \cite[Proposition 4.1]{mstdfcym} that $\nu(f)$ is the class of a fibre in an elliptic fibration on $W$ and, by Proposition \ref{prop:mirrorfibres}, we have $\check{L} \subset (\bZ f)_{\check{M}}^{\perp}\,/ \, \bZ f$. Thus $\nu(\beta).\nu(f) = 0$ for all $\beta \in \check{L}$.

This looks a lot like the definition of an $\check{L}$-polarization on a rational elliptic surface $(Y,D)$ of type $d$ (where we identify $\nu(f)$ with the class of the divisor $D$). This is not unexpected: the discussion from \cite[Section 4]{mstdfcym} suggests that the type II degeneration $\calV$ should be mirror to a ``slicing'' of the base $\bP^1$ of the fibration on $W$ into a series of (necessarily) analytic open sets, such that the fibration over each slice is mirror to a component of the central fibre $V_0$. In our setting, this gives the following conjecture.

\begin{conjecture} \label{conj:dht} There exists an analytic open set in $W$ which is isomorphic to an analytic open set $Y \setminus Z_D$ in an $\check{L}$-polarized rational elliptic surface $(Y,D)$ of type $d$, where $Z_D$ is a small closed neighbourhood of the elliptic fibre $D$.
\end{conjecture} 

\begin{remark} The philosophy of \cite{mstdfcym} suggests that the notion of mirror symmetry presented in this paper should admit a significant generalization to \emph{quasi-Fano pairs}: pairs $(X,C)$ consisting of a rational surface along with a smooth (but not necessarily nef and big) anticanonical divisor $C$. The pair $(X_2,E_1)$ constructed in the proof of Proposition \ref{prop:extension} is an example of such a quasi-Fano pair, that is not a weak del Pezzo pair (the anticanonical divisor $E_1$ is not nef). Most of the proofs in this section (with the notable exception of Proposition \ref{prop:extension}) require only minor modifications to work in this broader setting, and we would therefore expect a version  of Conjecture \ref{conj:dht} to hold. In this setting it is likely that the mirror will still be an elliptic fibration over an analytic open subset of $\bC$, but we do not expect it to always admit a compactification to a rational elliptic surface. We aim to address this in future work.
\end{remark}

\begin{remark} We briefly remark on a version of this conjecture in the setting of homological mirror symmetry. Homological mirror symmetry between the K3 surfaces $V_p$ and $W$ may be formulated as an equivalence between the category $\Perf(V_p)$ of perfect complexes on $V_p$ and the Fukaya category $\Fuk(W)$ of $W$. Harder and Katzarkov \cite{psca} have studied the extension of this equivalence to semistable degenerations of Type II, albeit in a much broader setting to the one considered here. In particular, they construct the category $\Perf(V_0)$ of perfect complexes on the semistable fibre $V_0$, which is a deformation of $\Perf(V_p)$, then present evidence for the conjecture that $\Perf(V_0)$ is equivalent to the derived exact Fukaya category of the complement of $(k-1)$ smooth fibres in the fibration on $W$ (where, as above, $k$ denotes the number of irreducible components in $V_0$).

To complete the picture note that, according to Seidel \cite{fcd}, the Fukaya category of $W$ with $(k-1)$ smooth fibres removed should be a deformation of the Fukaya category of $W$. We expect that Seidel's deformation of $\Fuk(W)$ should agree with the deformation of $\Perf(V_p)$ to $\Perf(V_0)$.
\end{remark}

\section{Compatibility with the Gross-Siebert program}\label{sec:gs}

In this section we will describe how the conjectures of Section \ref{sec:mirror} should fit with the Gross-Siebert mirror construction. We rely extensively on the methods from \cite{mslcysi,imspgwi} to perform these computations; we refer the interested reader to those papers for further details of the techniques used.

\subsection{Degree $8'$}\label{sec:8'}

We begin by describing in detail how this construction works in the degree $8'$ case, as this case is simple enough to be tractable but complex enough to display some interesting geometry. In this case $f_{8'}^{\perp} \subset \Lambda_{8'} \cong \mathrm{II}_{1,1}$ is generated by the root $a-b$ (here $a,b$ are generators of $\mathrm{II}_{1,1}$ as in Section \ref{sec:lattices}; they have $a^2 = b^2 = 0$ and $a.b = 1$). We therefore have two possible choices of lattice polarization: $L=0$ or $L = f_{8'}^{\perp} \cong A_1$. So that we can discuss the ideas of Section \ref{sec:complexkahler}, for now we assume $L = f_{8'}^{\perp}$; the case $L = 0$ will be discussed in Remark \ref{rem:generic8'}.

So let $(X,C)$ be an $L$-polarized weak del Pezzo pair of degree $8'$, with $L = f_{8'}^{\perp} \cong A_1$. Then $X \cong \bF_2$ and $C$ is a smooth elliptic curve. Let $s$ denote the class in $H^2(X,\bZ) \cong \Pic(X)$ of the unique $(-2)$-curve in $X$ and let $f$ denote the class of a fibre of the ruling on $X$. Then $s,f$ are primitive generators of $H^2(X,\bZ)$ and $s$ is a positive root which generates $L$. If we identify $s$ and $f$ with their Poincar\'{e} duals, then $s,f$ also generate the cone of effective curves on $X$.

\begin{remark}\label{rem:GSnonrigorous} To run the Gross-Siebert program and find the mirror of $(X,C)$, we first need to perform a \emph{toric degeneration} of $(X,C)$. However, even for simple toric degenerations, the computations required to proceed with the program from there are intractable with current methods. 

Instead, we travel by an easier route: we fix $X$ and degenerate $C$ to a union of four smooth rational curves $C' = C_1 \cup C_2 \cup C_3 \cup C_4$, find the Gross-Siebert mirror of the pair $(X,C')$ (which is very tractable with the methods of \cite{imspgwi}), then attempt to correct for the degeneration of $C$ at the end. Note that this is not a rigorous application of the Gross-Siebert program, so the mirror we produce here is \emph{not} the Gross-Siebert mirror of $(X,C)$ in a strict sense; we nonetheless hope that most important properties are shared between the two constructions.\end{remark}

Let $C' = C_1 \cup C_2 \cup C_3 \cup C_4 \subset X$ denote a degeneration of $C$ a union of four smooth rational curves, arranged in a square and ordered cyclically, chosen so that $[C_1] = s$, $[C_2] = [C_4] = f$, and $[C_3] = s + 2f$.

Next define $B$ to be the tropicalization of $X$, defined as follows. Let $\Div_{C'}(X) \subset \Div(X)$ denote the subspace of divisors supported on $C'$ and set $\Div_{C'}(X)_{\bR} := \Div_{C'}(X) \otimes_{\bZ} \bR$. Note that $\{C_i\}$ forms a basis of $\Div_{C'}(X)_{\bR}$. Let $\Div_{C'}(X)^* = \Hom(\Div_{C'}(X),\bZ)$ denote the dual lattice, $\Div_{C'}(X)_{\bR}^* := \Div_{C'}(X)^* \otimes_{\bZ} \bR$, and $\{C_i^*\}$ denote the dual basis of $\Div_{C'}(X)_{\bR}^*$. Define a collection of cones in $\Div_{C'}(X)_{\bR}^*$ by
\begin{align*}
\calP &:= \left\{\sum_{i \in I} \bR_{\geq 0} C_i^* \middle| I \subset \{1,2,3,4\} \text{ such that } \bigcap_{i \in I} C_i \neq \emptyset \right\} \\
&= \{\bR_{\geq 0} C_1^*,\ \bR_{\geq 0} C_2^*,\ \bR_{\geq 0} C_3^*,\ \bR_{\geq 0} C_4^*\ ,\bR_{\geq 0} C_1^*+\bR_{\geq 0} C_2^*,\\
& \quad \quad \bR_{\geq 0} C_2^*+\bR_{\geq 0} C_3^*,\  \bR_{\geq 0} C_3^*+\bR_{\geq 0} C_4^*,\  \bR_{\geq 0} C_1^*+\bR_{\geq 0} C_4^*\}.
\end{align*}
and set $B := \bigcup_{\tau \in \calP} \tau$. Finally, set $B(\bZ) := B \cap  \Div_{C'}(X)^*$.

Let $P$ denote the monoid in $H^2(X,\bZ)$ generated by the classes $s$ and $f$; note that $P$ contains all effective curve classes. We may identify $\bC[P]$ with $\bC[x,y]$ by identifying the class $ms + nf$ with the monomial $x^my^n$. Let $\mathfrak{m} = \langle x,y \rangle$ denote the maximal ideal, and set $I \subset \bC[x,y]$ to be any ideal with $\sqrt{I} = \mathfrak{m}$. Write $A_I = \bC[x,y]/I$ and set $R_I$ to be the free $A_I$-module
\[R_I := \bigoplus_{p \in B(\bZ)} A_I \vartheta_p.\]

We next define a multiplicative structure to make $R_I$ into an $A_I$-algebra. Define
\[\vartheta_p . \vartheta_q := \sum_{r \in B(\bZ)} \alpha_{pqr} \vartheta_r\]
for $\alpha_{pqr} \in A_I$ defined by 
\[\alpha_{pqr} = \sum_{\beta \in P \setminus I} N_{pqr}^{\beta}x^my^n,\]
where $\beta = ms + nf$ and $N_{pqr}^{\beta}$ is a count of curves in the class $\beta$.

In our case, $N_{pqr}^{\beta}$ is the number of rational curves $Z$ in class $\beta$ with marked points $x_1,x_2,x_3$ such that
\begin{itemize}
\item $Z$ is tangent to $C_i$ at $x_1$ with order $p(C_i)$;
\item $Z$ is tangent to $C_i$ at $x_2$ with order $q(C_i)$;
\item $Z$ is tangent to $C_i$ at $x_3$ with order $-r(C_i)$;
\item $x_3$ maps to a specified point in $\left(\cap_{i : r(C_i)>0}C_i\right)^{\circ}$ (where this intersection is $X$ if $r = 0$).
\end{itemize}

For instance, for the product $\vartheta_{C_2^*}.\vartheta_{C_{4}^*}$, if we set $r = \sum_{j=1}^4 r_j C_j^*$ we find
\begin{align*}\beta.C_k &= C_2^*(C_k) + C_{4}^*(C_k) - \sum_{j=1}^4r_jC_j^*(C_k) \\
&= \left\{\begin{array}{cl}  -r_k & \text{if } k = 1,3 \\ 1-r_k & \text{if } k = 2,4. \end{array}\right.
\end{align*}
But, if $\beta = ms + nf$, we also see that
\[\beta.C_k = \left\{ \begin{array}{cl} -2m + n & \text{if } k = 1 \\ m & \text{if } k = 2,4 \\ n & \text{if } k = 3. \end{array} \right.\]
Since $m,n \geq 0$ and $r_k \geq 0$ for all $k$,  it follows that we must have $n = 0$ and $m \in \{0,1\}$. However, if $m = 0$ then $r = C_2^* + C_4^*$, which is not a class in $B(\bZ)$. So we must have $m = 1$, which implies that $\beta = s$ and $r = 2C_1^*$.

$N_{C_2^*,C_4^*,2C_1^*}^s$ is the number of rational curves in the class $s = [C_1]$ which intersect $C_2$ and $C_4$ once each, have intersection number $(-2)$ with $C_1$, and pass through a specified point in $C_1^{\circ}$. There is clearly only one such curve (which is $C_1$ itself), so $N_{C_2^*,C_4^*,2C_1^*}^s = 1$ and $\alpha_{C_2^*,C_4^*,2C_1^*} = x$. Thus we find
\[\vartheta_{C_2^*}.\vartheta_{C_{4}^*} = x\, \vartheta_{2C_1^*}.\]

Similarly, with the convention $\vartheta_0 = 1$, we obtain
\begin{align*}
\vartheta_{C_1^*}.\vartheta_{C_{3}^*} &= y, \\
\vartheta_{C_i^*}.\vartheta_{C_{i}^*} &= \vartheta_{2C_i^*}, \\
\vartheta_{C_i^*}.\vartheta_{C_{i+1}^*} &= \vartheta_{C_i^* + C_{i+1}^*},
\end{align*}
where all indices are taken modulo $4$.

Finally, the Gross-Siebert mirror of $(X,C')$ is a general member of the family $\Spec\hat{R} \to \Spec\bC[x,y]$, where $\hat{R}$ is the inverse limit $\lim_{\leftarrow} R_I$ over ideals $I \subset \bC[x,y]$ with $\sqrt{I} = \mathfrak{m}$.  This is equal to
\[\xymatrix{\Spec\bC[x,y][\vartheta_{C_1^*},\vartheta_{C_2^*},\vartheta_{C_3^*},\vartheta_{C_4^*}]\,/\, (\vartheta_{C_1^*}\vartheta_{C_{3}^*} - y,\ \vartheta_{C_2^*}\vartheta_{C_{4}^*} - x\, \vartheta^2_{C_1^*}) \ar[d] \\
\Spec \bC[x,y].}\]
A general member of this family admits a fibration over $\bC$ (with parameter $t$) given by $t = \vartheta_{C_1^*}+\vartheta_{C_2^*}+\vartheta_{C_3^*}+\vartheta_{C_4^*}$. The general fibre in this fibration is an elliptic curve with four punctures.
\medskip

Now we come back to the fact that we began by degenerating $C$ to $C' = C_1 \cup C_2 \cup C_3 \cup C_4$ (see Remark \ref{rem:GSnonrigorous}). The philosophy of mirror symmetry states that smoothing the four nodes in $C'$ to obtain $C$ corresponds to filling in the four punctures in the elliptic fibres on the mirror, so we attempt to correct for our original degeneration $C \leadsto C'$ by compactifying the fibres in the mirror of $(X,C')$ to smooth elliptic curves.

This can be done explicitly as follows. Fix $x,y \in \bC$ generic and write a general member of the family above as
\begin{equation}\label{eq:F2}\{ac-y,\ bd-xa^2\} \subset \bA^4[a,b,c,d],\end{equation}
where we have set $a:= \vartheta_{C_1^*}$, $b:= \vartheta_{C_2^*}$, $c:= \vartheta_{C_3^*}$, and $d:= \vartheta_{C_4^*}$ for clarity. Compactifying to $\bP^4$, we obtain
\[\{ac-ye^2,\ bd-xa^2\} \subset \bP^4[a,b,c,d,e].\]
This surface has simple nodes ($A_1$ singularities) at the points $(0,1,0,0,0)$ and $(0,0,0,1,0)$, and an $A_3$ at the point $(0,0,1,0,0)$. These singularities should be resolved in the usual way to obtain a smooth surface $S$. 

The equation $a+b+c+d = te$ (for $t \in \bP^1$) defines a pencil of elliptic curves with $4$ base points on $S$. Blowing up these four base points, we obtain a rational elliptic surface $Y$ fibred over $\bP^1$ (with parameter $t$), which has an $\mathrm{I}_8$ singular fibre $D$ over $t = \infty$ and four singular fibres of type $\mathrm{I}_1$ over the points $t = \pm 2\sqrt{y(1 \pm 2 \sqrt{x})}$. We find that the pair $(Y,D)$ is a generic (unpolarized) rational elliptic surface of type $8'$. Removing $D$, we obtain an elliptic fibration over $\bC$, which should be thought of as mirror of our original pair $(X,C)$. As $\check{L} = 0$ in this setting, this agrees with the prediction of Conjecture \ref{conj:mirror}.
\medskip

Since $L = f_{8'}^{\perp}$ and $\check{L} = 0$, we may go further and ask how this fits with the correspondence between complex and K\"{a}hler moduli discussed in Section \ref{sec:complexkahler}. Suppose that we have fixed an embedding $L = f_{8'}^{\perp} \hookrightarrow F_{8'}^{\perp}$. Then we wish to examine the relationship between the K\"{a}hler cone of $X$ in a neighbourhood of the ray generated by $[C]$, and the strata in the complex moduli space of $Y$ in a neighbourhood of the stratum $U_{8',L}$.

We begin by studying the complex deformations of $Y$ given by the family \eqref{eq:F2}. Away from the lines $\{x = 0\}$ and $\{y = 0\}$, where the members of this family become non-normal, there is only one special sublocus where degenerations occur. This is the sublocus $\{x = \frac{1}{4}\}$, where two of the $\mathrm{I}_1$ fibres collide to produce a fibre of type $\mathrm{I}_2$. A member of the family \eqref{eq:F2} over a point in $\{x = \frac{1}{4}\}$ is a lattice polarized rational elliptic surface of type $8'$, polarized by the lattice $L \cong A_1$; such surfaces are all isomorphic. Indeed, the members of the family \eqref{eq:F2} over any line $\{x = \text{constant}, y \neq 0\}$ are all isomorphic, so the period map factors through the projection to the $y$-axis. The correspondence with the stratification \eqref{eq:strata} here is clear: $\calP_{8',0}$ corresponds to the entire $(x,y)$-plane (without the two lines $\{x=0\}$ and $\{y=0\}$), and $U_{8',L}$ corresponds to the sublocus $\{x = \frac{1}{4}\}$.

Next we look at the nef cone of $X$. This cone is generated by the two classes $[C] = s+2f$ and $f$, which lie along its boundary rays. It is dual to the effective cone, which is generated by the classes $s$ and $f$.

Now we make an important observation. The ray generated by $[C]$ in the nef cone is dual to the ray generated by $s$ in the effective cone. Let $P' \subset P$ denote the submonoid obtained by setting the coefficient of $s$ to zero (so $P'$ is generated by multiples of $f$). Identifying $\bC[P]$ with $\bC[x,y]$ and taking Spec, we find that the submonoid $P'$ gives rise to a sublocus $\{x = \text{constant}\}$ in $\bA^2[x,y]$. This is precisely the form observed for the sublocus corresponding to the lattice enhancement above. This idea will be discussed further in Section \ref{sec:GSconjecture}.

\begin{remark}\label{rem:generic8'} One may ask what happens if we perform the same calculation for a generic (unpolarized) weak del Pezzo pair $(X,C)$ of degree $8'$. Here $X = \bP^1 \times \bP^1$ and $C$ is a smooth anticanonical divisor. For a suitable degeneration $C \leadsto C'$, the Gross-Siebert mirror of $(X,C')$ is the family over $\Spec \bC[x,y]$ given by
\begin{equation}\label{eq:generic8'} \{ac - y,\ bd - x\} \subset \bA^4[a,b,c,d].\end{equation}
Compactifying to $\bP^4$, we obtain
\[\{ac-ye^2,\ bd-xe^2\} \subset \bP^4[a,b,c,d,e].\]
This surface has four simple nodes at $(1,0,0,0,0)$, $(0,1,0,0,0)$, $(0,0,1,0,0)$, and $(0,0,0,1,0)$; these should be resolved to obtain a smooth surface $S$. 

The equation $a+b+c+d = te$ (for $t \in \bP^1$) defines a pencil of elliptic curves with $4$ base points on $S$. Blowing up these four base points, we obtain a rational elliptic surface $Y$ fibred over $\bP^1$ (with parameter $t$), which has an $\mathrm{I}_8$ singular fibre $D$ over $t = \infty$ and four singular fibres of type $\mathrm{I}_1$ at the points $t = \pm 2\sqrt{x} \pm 2 \sqrt{y}$. As above, the pair $(Y,D)$ obtained is a generic rational elliptic surface of type $8'$. Removing $D$, we obtain an elliptic fibration over $\bC$, with four singular fibres of type $\mathrm{I}_1$. 

This is \emph{not} what is predicted by Conjecture \ref{conj:mirror}, which claims that the mirror should be polarized by the lattice $f_{8'}^{\perp} \cong A_1$. Such a polarization should give rise to a fibration over $\bC$ with two fibres of type $\mathrm{I}_1$ and one fibre of type $\mathrm{I}_2$. 

However, we may salvage something: $f_{8'}^{\perp}$-polarized rational elliptic surfaces appear as the subfamily $x = y$ in the family \eqref{eq:generic8'}. We postulate that it is this \emph{subfamily} which is the mirror to $(X,C)$ in the lattice polarized sense.

This appears to be a general phenomenon. Indeed, in all cases with $R_{\check{L}} \neq 0$ that we have explicitly computed, the mirror family predicted by Conjecture \ref{conj:mirror} appears as a subfamily of the computed mirror family.\end{remark}

\subsection{Degree $8$}\label{sec:8}

It is instructive here to contrast the case of a weak del Pezzo pair $(X,C)$ of degree $8$. In this case, $X$ is always isomorphic to the Hirzebruch surface $\bF_1$ (independent of the choice of lattice polarization) and $C$ is a smooth anticanonical divisor. Running the Gross-Siebert program for a suitable degeneration $C \leadsto C'$, as above, we obtain a family over $\bA^2[x,y]$ given by 
\begin{equation}\label{eq:generic8} \{ac - y,\ bd - xa\} \subset \bA^4[a,b,c,d],\end{equation}
which is the Gross-Siebert mirror to the pair $(X,C')$. Compactifying to  $\bP^4$, we obtain
\[\{ac-ye^2,\ bd-xae\} \subset \bP^4[a,b,c,d,e].\]
This surface has $A_1$ singularities at $(0,1,0,0,0)$ and $(0,0,0,1,0)$, and an $A_2$ at $(0,0,1,0,0)$; these should be resolved to obtain a smooth surface $S$. 

The equation $a+b+c+d = te$ (for $t \in \bP^1$) defines a pencil of elliptic curves with $4$ base points on $S$. Blowing up these four base points, we obtain a rational elliptic surface $Y$ fibred over $\bP^1$ (with parameter $t$), which has an $\mathrm{I}_8$ singular fibre $D$ over $t = \infty$ and four singular fibres of type $\mathrm{I}_1$ at the roots of the polynomial
\[t^4 + xt^3-8yt^2-36xyt-27x^2y+16y^2.\] 
In this case, the pair $(Y,D)$ is a generic rational elliptic surface of type $8$. Removing $D$, we obtain an elliptic fibration over $\bC$, which should be thought of as mirror to our original pair $(X,C)$.
\medskip

How does this fit with lattice polarizations? Recall that $f_8 = 3l - e_1 \in \mathrm{I}_{1,1}$, so $f_8^{\perp}$ is generated by $l - 3e_1$, which has self-intersection $-8$ (see Table \ref{tab:latticetype}). Thus $f_8^{\perp}$ does not contain any roots and $R_L = \{0\}$ independent of which polarizing lattice $L$ we take. Proceeding to the mirror, the fact that $f_8^{\perp}$ does not contain any roots also implies that $R_{\check{L}}$ must be trivial, so the mirror to $(X,C)$ should be the open set $Y \setminus D$ in a generic rational elliptic surface $(Y,D)$ of type $8$. This is exactly what is given by equation \eqref{eq:generic8}.

This is not quite the end of the story for degree $8$. Indeed, one may ask whether the rational elliptic surface defined by equation \eqref{eq:generic8} can ever acquire fibres that are worse than $\mathrm{I}_1$ away from $t = \infty$, as happened along the sublocus $\{x = \frac{1}{4}\}$ in the degree $8'$ case? Indeed, examining the discriminant of the quartic equation determining $t$, one sees that such fibres occur if and only if $256y + 27x^2 = 0$. Imposing this condition, we obtain a rational elliptic surface with two fibres of type $\mathrm{I}_1$ and one fibre of type $\mathrm{II}$. But a type $\mathrm{II}$ fibre is irreducible, so does not give an enhancement in the lattice polarization.

\subsection{A compatibility conjecture} \label{sec:GSconjecture}

Based on computations like these, we conclude this paper with a conjecture that relates the notion of mirror symmetry introduced in Section \ref{sec:mirror} to the Gross-Siebert program.

We begin with some setup. Let $(X,C,\mu)$ denote a marked $L$-polarized weak del Pezzo pair of degree $d$, with $L = f_d^{\perp}$. Let $P$ denote the effective cone of $X$. Then  by \cite[Proposition 6.2]{dhmq}, $P$ is rational polyhedral and generated by a finite set of rays; let $r_1,\ldots,r_n$ denote generators of these rays, for some $n \in \bN$.

%Now let $S$ denote a codimension $c$ face of the nef cone of $X$, and suppose that $S$ contains the class $[C]$. Then $S = F_1 \cap \ldots \cap F_c$, for $F_i$ codimension $1$ faces containing $[C]$. 

Let $\{\alpha_1,\ldots,\alpha_k\}$ denote the set of simple roots in $\Phi_L^+$. By Proposition \ref{prop:dPKahler}, each face $F_i$ of $\Nef(X)$ with $[C] \subset F_i$ is orthogonal to $\mu^{-1}(\alpha_i)$, for some $i \in \{1,\ldots,k\}$. Thus, by duality, each class $\mu^{-1}(\alpha_i)$ generates one of the boundary rays $r_j$ in the effective cone. Without loss of generality, assume that $\mu^{-1}(\alpha_i) = r_i$, so that $[C]$ is orthogonal to the face of the effective cone generated by $\{r_1,\ldots,r_k\}$.

In this setting, the computation of Section \ref{sec:8'} constructs a mirror family $\calY$ over $\Spec\bC[P]$. Let $x_i$ denote the variable in $\bC[P]$ corresponding to $r_i$. Based on the the mirror correspondence postulated in Section \ref{sec:complexkahler} and the computations of Sections \ref{sec:8'}  and \ref{sec:8}, one might expect there to exist codimension $1$ subfamilies $\calY_i := \{x_i = c_i\}$ of $\calY$, for $c_i$ constants and $i \in \{1,\ldots,k\}$, so that $\calY_i$ corresponds to the codimension $1$ stratum $U_{d,\langle \alpha_i \rangle} \subset \calP_{d,0}$ (here we have assumed that an embedding $L = f_d^{\perp} \hookrightarrow F_d^{\perp}$ has been chosen, so $\langle \alpha_i \rangle \subset L$ can be identified with an overlattice of $\check{L} = 0 \subset F_d^{\perp}$).

Unfortunately this is not quite the case. The mirror correspondence in Section \ref{sec:complexkahler} is fundamentally local in nature; when we attempt to extend it to the constructed mirror family $\calY$ we find that we cannot, in general, make a global choice of marking, so we cannot directly identify subfamilies of $\calY$ with strata in the (marked) period domain $\calP_{d,0}$. Instead, we have to compensate for the possible action of monodromy on the markings.

By \cite[Lemma 4.3]{msacc}, this monodromy group is contained in the group $\mathrm{Adm}_d$ of \emph{$d$-admissible automorphisms} of $\mathrm{I}_{1,9}$.

\begin{definition}An automorphism of $\mathrm{I}_{1,9}$ is \emph{$d$-admissible} if it fixes the sublattice $F_d$ and preserves the cone $\calC_d^{++}$.
\end{definition}

\begin{remark} The results of \cite[Section 6]{msacc} show that the stack quotient $\calP_{d,0}/\mathrm{Adm}_d$ is a period domain for \emph{unmarked} rational elliptic surfaces of degree $d$, and the corresponding period mapping is an isomorphism over an open set in this space.
\end{remark}

The group $\mathrm{Adm}_d$ acts on the set $\{\alpha_1,\ldots,\alpha_k\}$ of simple roots in $F_d^{\perp}$. For each $\alpha_i \in \{\alpha_1,\ldots,\alpha_k\}$, define $\calA_d(\alpha_i)$ to be the set of $\alpha_j \in \{\alpha_1,\ldots,\alpha_k\}$ such that $\alpha_j = \varphi(\alpha_i)$ for some $\varphi \in \mathrm{Adm}_d$. The set $\calA_d(\alpha_i)$ should be thought of as those simple roots $\alpha_j \in \Phi^+_L$ that are equivalent to $\alpha_i$ up to the choice of marking. Similarly, define $\calA_d(r_i)$ to be $\mu^{-1}(\calA_d(\alpha_i))$ and $\calA_d(x_i)$ to be the set of variables in $\bC[P]$ corresponding to $\calA_d(r_i)$.

Then we have the following conjecture, which should be thought of as a formalization of the discussion in Section \ref{sec:complexkahler}.

\begin{conjecture}\label{conj:GSconj} For each face $F_i$ of $\Nef(X)$ with $[C] \subset F_i$, there exists a codimension $1$ subfamily $\calY_i$ of the mirror family $\calY$ \textup{(}constructed using the method of Section \ref{sec:8'}\textup{)} whose members are $\langle \alpha_i \rangle$-polarized rational elliptic surfaces of type $d$. Explicitly, $\calY_i$ is given by a subfamily $f_i(\calA_d(x_i)) = 0$, where $f_i(A_d(x_i))$ is a polynomial in the variables from the set $\calA_d(x_i)$.

Moreover, if $S = \bigcap_{i \in I} F_i$ is a subface of $\Nef(X)$, then the subfamily $\calY_S := \bigcap_{i \in I} \calY_i$ is a family of $\langle \alpha_i : i \in I \rangle$-polarized rational elliptic surfaces of type $d$.
\end{conjecture}

\begin{remark} We note that the subfamilies $\calY_i$ will not, in general, all be distinct. Indeed, we should have $\calY_i = \calY_j$ if and only if $\calA_d(\alpha_i) = \calA_d(\alpha_j)$. In this case, one should think of the two faces $F_i$ and $F_j$ of $\Nef(X)$ as corresponding to two different markings on the surfaces in the subfamily $\calY_i$.

We also note that if $\calY_i = \calY_j$, then an intersection $\calY_i \cap \calY_j$ should be interpreted as a singular locus in $\calY_i$ of codimension $1$.
\end{remark}

\begin{remark} The authors have verified this conjecture explicitly for degrees $d \in \{6,7,8,8',9\}$.
\end{remark}

\begin{remark}One might hope that a similar conjecture would hold to describe the structure of the Gross-Siebert mirror of $(X,C)$ (see Remark \ref{rem:GSnonrigorous}). However, the difficulty of explicitly computing Gross-Siebert mirrors means that the authors have been unable to verify this in even the simplest cases.
\end{remark}

\begin{example} This conjecture is best illustrated by an example. Consider a marked $L$-polarized weak del Pezzo pair $(X,C,\mu)$ of degree $6$, with $L = f_6^{\perp} \cong A_2 + A_1$. The lattice $L$ contains three simple roots $\alpha_1,\alpha_2,\alpha_3 \in \Phi^+_L$, and the effective cone $P$ of $X$ is generated by the $\alpha_i$ and an additional class $E$, with intersection matrix
\[\begin{array}{c|cccc} & E & \alpha_1 & \alpha_2 & \alpha_3 \\
\hline
E & -1 & 1 & 1 & 0 \\
\alpha_1 & 1 & -2 & 0 & 0 \\
\alpha_2 & 1 & 0 & -2 & 1 \\
\alpha_3 & 0 & 0 & 1 & -2  
\end{array}
\]

By Proposition \ref{prop:dPKahler}, in a neighbourhood of the ray generated by $[C] = [-K_X]$, the codimension $1$ faces of $\Nef(X)$ are given by the three hyperplanes $F_i := (\mu^{-1}(\alpha_i))^{\perp}$. 

Identify $\bC[P]$ with $\bC[x,y,z,t]$ by identifying the class $k\alpha_1 + l\alpha_2 + m \alpha_3 + n E$ with $x^ky^lz^mt^n$. The construction of Section \ref{sec:8'} gives rise to a family $\calY$ over $\Spec \bC[x,y,z,t]$, whose generic member is an (unpolarized) rational elliptic surface of degree $6$. Then we find the following correspondence between subfamilies and the structure of $\Nef(X)$.
\begin{itemize}
\item The locus $\{x = \frac{1}{4}\}$ gives rise to a subfamily $\calY_1$ of rational elliptic surfaces with fibre type $\mathrm{I}_6\mathrm{I}_2\mathrm{I}_1\mathrm{I}_1\mathrm{I}_1\mathrm{I}_1$. These admit a polarization by the lattice $\langle \alpha_1 \rangle \cong A_1$. Noting that $\calA_6(\alpha_1) = \{\alpha_1\}$, we see that $\calA_6(x) = \{x\}$, so $f_1(\calA_6(x)) = x - \frac{1}{4}$. In the sense of Conjecture \ref{conj:GSconj}, this subfamily is mirror to the face $F_1$.

\item The locus $\{27y^2z^2 - 18yz + 4y + 4z - 1 = 0\}$ gives rise to a subfamily $\calY_2$ of rational elliptic surfaces, which also have fibre type $\mathrm{I}_6\mathrm{I}_2\mathrm{I}_1\mathrm{I}_1\mathrm{I}_1\mathrm{I}_1$. These admit a polarization by the lattice $\langle \alpha_2 \rangle \cong A_1$. Observing that $\calA_6(\alpha_2) = \calA_6(\alpha_3) = \{\alpha_2,\alpha_3\}$, we find that $\calA_6(y) = \calA_6(z) = \{y,z\}$, so $f_2(\calA_6(y)) = f_3(\calA_6(z)) = 27y^2z^2 - 18yz + 4y + 4z - 1$.  In the sense of Conjecture \ref{conj:GSconj}, this subfamily is mirror to the faces $F_2$ and $F_3$.

\item The sublocus $\{x - \frac{1}{4} = 27y^2z^2 - 18yz + 4y + 4z - 1 = 0\}$, corresponding to the intersection $\calY_1 \cap \calY_2$, gives rise to a subfamily $\calY_{12}$ of rational elliptic surfaces, which have fibre type $\mathrm{I}_6\mathrm{I}_2\mathrm{I}_2\mathrm{I}_1\mathrm{I}_1$. These admit a polarization by the lattice $\langle \alpha_1,\alpha_2 \rangle \cong A_1+A_1$.  In the sense of Conjecture \ref{conj:GSconj}, this subfamily is mirror to the faces $F_1 \cap F_2$ and $F_1 \cap F_3$.

\item The locus $\{27y^2z^2 - 18yz + 4y + 4z - 1 = 0\}$ is singular in codimension $1$, along the sublocus $\{y = z = \frac{1}{3}\}$. Over this sublocus, we have a subfamily $\calY_{22}$ of rational elliptic surfaces which have fibre type $\mathrm{I}_6\mathrm{I}_3\mathrm{I}_1\mathrm{I}_1\mathrm{I}_1$. These admit a polarization by the lattice $\langle \alpha_2,\alpha_3 \rangle \cong A_2$. In the sense of Conjecture \ref{conj:GSconj}, this subfamily is mirror to the face $F_2 \cap F_3$.

\item Finally, we have the sublocus $\{x - \frac{1}{4} = y- \frac{1}{3} = z-\frac{1}{3} = 0\}$.  Over this sublocus we have a subfamily $\calY_{122}$ of rational elliptic surfaces which have fibre type $\mathrm{I}_6\mathrm{I}_3\mathrm{I}_2\mathrm{I}_1$. These admit a polarization by the lattice $L = \langle \alpha_1, \alpha_2,\alpha_3 \rangle \cong A_2 + A_1$. In the sense of Conjecture \ref{conj:GSconj}, this subfamily is mirror to the face $F_1 \cap F_2 \cap F_3$, which is the ray generated by $[C]$.
\end{itemize}
\end{example}

\end{document}